\documentclass[reqno,12pt]{amsart}
\usepackage[colorlinks=true, linkcolor=blue, citecolor=blue]{hyperref}

\usepackage{amssymb}
\usepackage{amsmath, graphicx, rotating}
\usepackage{color}
\usepackage{soul}
\usepackage[dvipsnames]{xcolor}

\allowdisplaybreaks
\usepackage{ifthen}
\usepackage{xkeyval}
\usepackage{todonotes}
\usepackage[T1]{fontenc}
\usepackage{lmodern}
\usepackage[english]{babel}

\usepackage{ upgreek }
\usepackage{stmaryrd}
\SetSymbolFont{stmry}{bold}{U}{stmry}{m}{n}
\usepackage{amsthm}
\usepackage{float}

\usepackage{ bbm }
\usepackage{ stmaryrd }
\usepackage{ mathrsfs }
\usepackage{ frcursive }
\usepackage{ comment }

\usepackage{pgf, tikz}
\usetikzlibrary{shapes}
\usepackage{varioref}
\usepackage{enumitem}

\definecolor{rouge}{rgb}{0.7,0.00,0.00}
\definecolor{vert}{rgb}{0.00,0.5,0.00}
\definecolor{bleu}{rgb}{0.00,0.00,0.8}
\usepackage[margin=1in]{geometry}
\newtheorem{theorem}{Theorem}[section]
\newtheorem*{theorem*}{Theorem}
\newtheorem{lemma}[theorem]{Lemma}
\newtheorem{definition}[theorem]{Definition}

\newtheorem{proposition}[theorem]{Proposition}

\labelformat{hypothesis}{\textbf{M\kern-0.1mm#1}}

\newtheorem{condition}{Condition}

\newtheorem{conditionA}{A\kern-0.1mm}
\labelformat{conditionA}{\textbf{A\kern-0.1mm#1}}

\newtheorem{conditionB}{B\kern-0.1mm}
\labelformat{conditionB}{\textbf{B\kern-0.1mm#1}}

\theoremstyle{definition}

\newtheorem{remark}[theorem]{Remark}

\def \eref#1{\hbox{(\ref{#1})}}

\numberwithin{equation}{section}

\def\geq{\geqslant}
\def\leq{\leqslant}

\def\RR{\mathbb{R}}
\def\PP{\mathbb{P}}
\def\EE{\mathbb{E}}

\def\vare{{\varepsilon}}
\def \eref#1{\hbox{(\ref{#1})}}

\def\EE{\mathbb{ E}}

\begin{document}

\title[Well-posedness and averaging principle of McKean-Vlasov SPDEs ]
{Well-posedness and averaging principle of McKean-Vlasov SPDEs driven by cylindrical $\alpha$-stable process}

\author{Mengyuan Kong}
\curraddr[Kong, M.]
{School of Mathematics and Statistics, Jiangsu Normal University, Xuzhou, 221116, China}
\email{mykong@jsnu.edu.cn}

\author{Yinghui Shi}
\curraddr[Shi, Y.]
{School of Mathematics and Statistics and Research Institute of Mathematical Science, Jiangsu Normal University, Xuzhou, 221116, China}
\email{shiyinghui@jsnu.edu.cn}

\author{Xiaobin Sun}
\curraddr[Sun, X.]{ School of Mathematics and Statistics and Research Institute of Mathematical Science, Jiangsu Normal University, Xuzhou, 221116, China}
\email{xbsun@jsnu.edu.cn}

\begin{abstract}
In this paper, we first study the well-posedness of a class of McKean-Vlasov stochastic partial differential equations driven by cylindrical $\alpha$-stable process, where $\alpha\in(1,2)$. Then by the method of the Khasminskii's time discretization, we prove the averaging principle of a class of multiscale McKean-Vlasov stochastic partial differential equations driven by cylindrical $\alpha$-stable processes. Meanwhile, we obtain a specific strong convergence rate.
%which extents the model considered in \cite{BYY} to the model of the distribution dependent coefficients.
\end{abstract}

\date{\today}

\maketitle

\section{Introduction}

Let $H$ be a Hilbert space with the inner product $\langle\cdot,\cdot\rangle$ and the norm $|\cdot|$. Let $\mathscr{P}$ be the set of all probability measures on $(H, \mathscr{B}(H))$. For any $p\geq 1$, define
$$
\mathscr{P}_p:=\Big\{\mu\in \mathscr{P}: \mu(|\cdot|^p):=\int_{H}|x|^p\mu(dx)<\infty\Big\}.
$$
Then space $(\mathscr{P}_p, \mathbb{W}_p)$ is a complete metric space, where $\mathbb{W}_p$ is the $L^p$-Wasserstein distance, i.e.,
$$
\mathbb{W}_p(\mu_1,\mu_2):=\inf_{\pi\in \mathscr{C}_{\mu_1,\mu_2}}\left[\int_{H\times H}|x-y|^p\pi(dx,dy)\right]^{1/p}, \quad \mu_1,\mu_2\in\mathscr{P}_p,
$$
where $\mathscr{C}_{\mu_1,\mu_2}$ is the set of all couplings for $\mu_1$ and $\mu_2$.

\vspace{0.1cm}
We first consider the following McKean-Vlasov stochastic partial differential equations (SPDEs for short) in $H$:
\begin{equation}\label{main equation0}\left\{\begin{array}{l}
\displaystyle dX_t=AX_tdt+B(X_t,\mathscr{L}_{X_t})dt+dL_t,\\
X_0=\xi\in H,\end{array}\right.
\end{equation}
where $\mathscr{L}_{X_t}$ is the law of random variable $X_t$, $A:\mathcal{D}(A)\subset H\to H$ is a self-adjoint operator, which is the infinitesimal generator of a linear strongly continuous
semigroup $(e^{tA})_{t\ge0}$. The process $L=(L_t)_{t\geq0}$ is a cylindrical $\alpha$-stable process with $\alpha\in(1,2)$ defined on a probability space $(\Omega,\mathcal{F},\mathbb{P})$ with filtration $(\mathcal{F}_t)_{t\geq 0}$ . $\xi$ is an $H$-valued $\mathcal{F}_0$-measurable random variable, map $B: H\times\mathscr{P}_p \rightarrow H$ satisfies proper condition.

\vspace{0.1cm}
The McKean-Vlasov stochastic differential equations (SDEs for short), also called distribution dependent SDEs, describe stochastic
systems whose evolution is determined by both the microcosmic location and the macrocosmic distribution of the particle, see \cite{Mc}.  When the noise is the classical Brownian motion, the well-posedness of such kind of stochastic equations have been studied intensively (see e.g. \cite{HW2019, HY, MV, RZ, WFY} for finite dimension and  e.g. \cite{AD, GA, HL,HS2021,RW} for infinite dimension). Further properties of the solution, such as large deviations, ergodicity, Harnack inequality and the Bismut formula for the Lions Derivative, which also have been investigated in many references (see e.g. \cite{BRW, LMW, RW, RW2019, Song}).

\vspace{0.1cm}
However, to the authors' knowledge, it seems no result about the well-posedness of the  McKean-Vlasov SPDEs driven by $\alpha$-stable process. Meanwhile, considering the cylindrical $\alpha$-stable process has theoretically meaningful, for instance such kind of processes can be used to model systems with heavy tails in physics. Hence, the first purpose of this paper is study the well-posedness of equation \eref{main equation0}. The main method is based on the classical contraction mapping principle.

\vspace{0.1cm}
Our second purpose is further to study the averaging principle for the following multiscale McKean-Vlasov SPDEs driven by cylindrical $\alpha$-stable processes:
\begin{equation}\left\{\begin{array}{l}\label{MSEQ}
\displaystyle
dX^{\vare}_t=\left[AX^{\vare}_t+F(X^{\vare}_t, \mathscr{L}_{X^{\vare}_t}, Y^{\vare}_t)\right]dt+dL_t,\quad X^{\vare}_0=\xi\in H,\\
dY^{\vare}_t=\frac{1}{\vare}[AY^{\vare}_t+G(X^{\vare}_t, \mathscr{L}_{X^{\vare}_t},Y^{\vare}_t)]dt+\frac{1}{\vare^{1/\alpha}}dZ_t,\quad Y^{\vare}_0=\eta\in H,\end{array}\right.
\end{equation}
where $\varepsilon >0$ is a small parameter describing the ratio of time scales between the slow component $X^{\varepsilon}$
and fast component $Y^{\varepsilon}$. Measurable functions $F,G:H\times\mathscr{P}_p\times H\rightarrow H$ satisfy some appropriate conditions, and $\{L_t\}_{t\geq 0}$ and $\{Z_t\}_{t\geq 0}$ are mutually independent cylindrical $\alpha$-stable process with $\alpha\in (1,2)$ defined on a filtered probability space $(\Omega,\mathcal{F},(\mathcal{F}_t)_{t\geq 0},\mathbb{P})$. $\xi$ and $\eta$ are two $H$-valued $\mathcal{F}_0$-measurable random variables.

\vspace{0.1cm}
The averaging principle of stochastic system \eref{MSEQ} is to describe the asymptotic behavior of the slow component $X^{\vare}_t$ as $\vare\to 0$, which says that the slow component will converge to the so-called averaged equation in various senses. In this paper we intend to prove the slow component $X^{\varepsilon}$ convergent to $\bar{X}$ in the strong sense, i.e., for any initial values $\xi$ and $\eta$ which are two $H$-valued random variables having finite $k$-th moment for $m\in [1,\alpha)$, one tries
to find a constant  $r>0$  such that  for $T>0$
\begin{align}
\left[\mathbb{E}\left(\sup_{t\in [0,T]}|X_{t}^{\vare}-\bar{X}_{t}|^{m}\right)\right]^{1/m}\leq C\vare^{r}, \label{Mainresult}
\end{align}
where  $C$ is a positive constant only depends on $T$, $\xi$ and $\eta$, and $\bar{X}$ is the solution of the corresponding averaged equation (see  equation \eref{1.3} below).

\vspace{0.1cm}
The theory of averaging principle was first developed for the  ordinary differential equations by Bogoliubov and Mitropolsky \cite{BM}, and extended to SDEs  by Khasminskii \cite{K1} and SPDEs by Cerrai and Freidlin \cite{CF}. Nowadays, the averaging principle for slow-fast stochastic system has been drawn much attentions, see e.g. \cite{B1,C1,DSXZ,FLL,GP4,GD,LRSX1,PS,PXY,RX,V0,WR,XML1} and the references therein.

\vspace{0.1cm}
In the distribution-independent case, the averaging principle for multiscale SDEs driven by $\alpha$-stable processes has been studied in a number of papers.
For instance, Bao et al. \cite{BYY} prove the strong averaging principle for two-time scale SPDEs driven by $\alpha$-stable processes. Chen et al. \cite{CSS} prove the strong averaging principle for stochastic Burgers equations. Sun et al. \cite{SXX} prove the strong and weak convergence rates for a class of multiscale SDEs driven by $\alpha$-stable processes. Sun and Zhai \cite{SZ} prove the strong averaging principle for stochastic Ginzburg-Landau equation. Sun et al. \cite{SXXZ} prove the strong averaging principle for a class of singular SPDEs
driven by $\alpha$-stable processes.

\vspace{0.1cm}
In the distribution dependent case, it seems there are few results. For example, R\"{o}ckner et al. \cite{RSX} study the averaging principle for a class of slow-fast McKean-Vlasov
stochastic differential equations driven by Wiener noise. However unlike the Wiener noise, the cylindrical  $\alpha$-stable process only has finite $p$-th moment for $p\in(0,\alpha)$, thus some methods developed in \cite{RSX} are not suitable to treat the cylindrical $\alpha$-stable noises, therefore we require new and different techniques to deal with the cylindrical $\alpha$-stable noise.
%Meanwhile, the cylindrical  $\alpha$-stable process has numerous applications in physics because such processes can be used to model systems with heavy tails.

\vspace{0.1cm}
Note that the exponential ergodicity of the transition semigroup of the corresponding frozen equation plays an important role in the proof of strong averaging convergence.
However it does not hold if the coefficient $G$ in the fast component depends on the law of the fast
component. In this situation, the corresponding frozen equation is also a McKean-Vlasov SPDEs. As a result, the corresponding Markov operator of the frozen equation $P_tf(y):=\EE f(Y^{x,\mu,y}_t)$ is not a semigroup anymore (see \cite[(1.11)]{WFY}), where $\{Y^{x,\mu,y}_t\}_{t\geq 0}$ denotes the unique solution of the frozen equation by fixed $x\in H,\mu\in \mathscr{P}_p$. Then it will brings some essential difficulties. Consequence, we focus on the coefficients depending only on the law of the slow component here.
%The coefficients depending  on the law of the fast component is left for our further work.

\vspace{0.1cm}
The paper is organized as follows. In section 2, we study the well-posedness of a class of McKean-Vlasov SPDEs driven by $\alpha$-stable processes. In sections 3, we study the averaging principle for a class of multiscale Mckean-Vlasov SPDEs. Section 4 is the appendix.

\vspace{1mm}
We note that throughout this paper $C$ and $C_T$  denote positive constants which may change from line to line, where the subscript $T$ is used to emphasize that the constant depends on $T$.

\section{Well-posedness of McKean-Vlasov SPDEs driven by $\alpha$-stable process}

In this section, we first give some notations and assumptions. Then we prove the existence and uniqueness of the mild solution of a class of McKean-Vlasov SPDEs driven by $\alpha$-stable processes.

\subsection{Notations and assumptions}

We recall the main equation
\begin{equation}\label{main equation}\left\{\begin{array}{l}
\displaystyle dX_t=AX_tdt+B(X_t,\mathscr{L}_{X_t})dt+dL_t,\\
X_0=\xi\in H,\end{array}\right.
\end{equation}
where $\{L_t\}_{t\geq 0}$ is a cylindrical $\alpha$-stable process with $\alpha\in(1,2)$, which is given by
$$L_t=\sum^{\infty}_{k=1}\beta_{k}L^{k}_{t}e_k,\quad t\geq 0$$
where $\{e_k\}^{\infty}_{k=1}$ is a complete orthonormal basis of $H$, $\{\beta_k\}^{\infty}_{k=1}$ is a given sequence of positive numbers and $\{L^{k}_t\}^{\infty}_{k=1}$ is a sequence of independent one dimensional symmetric $\alpha$-stable processes satisfies for any $k\geq 1$ and $t\geq0$,
$$\mathbb{E}[e^{iL^k_{t}h}]=e^{-t|h|^{\alpha}}, \quad h\in \mathbb{R}.$$

We suppose that the following conditions hold:

\begin{conditionA}\label{A1} $A$ is a selfadjoint operator which satisfies $Ae_k=-\lambda_k e_k$ with $\lambda_k>0$ and $\lambda_k\uparrow \infty$, as $k\uparrow \infty$, where $\{e_k\}^{\infty}_{k=1}\subset \mathscr{D}(A)$.
\end{conditionA}

\begin{conditionA}\label{A2}
Suppose that there exists a constant $p\in [1,\alpha)$ such that
$$
B(\cdot,\cdot):H\times \mathscr{P}_p\rightarrow H.
$$
Furthermore, there exists $C>0$ such that for any $x,y\in H$, $\mu,\nu\in \mathscr{P}_p$,
\begin{eqnarray*}
\left|B(x, \mu)-B(y,\nu)\right|\leq C\left[|x-y|+\mathbb{W}_p(\mu,\nu)\right].
\end{eqnarray*}
\end{conditionA}

\begin{conditionA}\label{A3} Assume that $\sum^{\infty}_{k=1}\frac{\beta^{\alpha}_k}{\lambda_k}<\infty$.
\end{conditionA}

For any $s\in\RR$, we define
 $$H^s:=\mathscr{D}((-A)^{s/2}):=\left\{u=\sum^{\infty}_{k=1}u_ke_k: u_k\in \mathbb{R},~\sum^{\infty}_{k=1}\lambda_k^{s}u_k^2<\infty\right\}$$
and
 $$(-A)^{s/2}u:=\sum^{\infty}_{k=1}\lambda_k^{s/2} u_ke_k,\quad u\in\mathscr{D}((-A)^{s/2}),$$
with the associated norm $\|u\|_{s}:=|(-A)^{s/2}u|=\left(\sum^{\infty}_{k=1}\lambda_k^{s} u^2_k\right)^{1/2}$. It is easy to see $\|\cdot\|_0=|\cdot|$.

The following smoothing properties of the semigroup $e^{tA}$ (see \cite[Proposition 2.4]{B1}) will be used quite often later in this paper:
\begin{eqnarray}
&&|e^{tA}x|\leq e^{-\lambda_1 t}|x|, t\geq 0,\label{P2}\\
&&\|e^{tA}x\|_{\sigma_2}\leq C_{\sigma_1,\sigma_2}t^{-\frac{\sigma_2-\sigma_1}{2}}e^{-\frac{\lambda_1 t}{2}}\|x\|_{\sigma_1},\quad x\in H^{\sigma_2},\sigma_1\leq\sigma_2, t>0,\label{P3}\\
&&|e^{tA}x-x|\leq C_{\sigma}t^{\frac{\sigma}{2}}\|x\|_{\sigma},\quad x\in H^{\sigma},\sigma>0, t\geq 0.\label{P4}
\end{eqnarray}

\subsection{The proof of well-posedness of equation \eref{main equation}}

In this subsection, we shall prove the well-posedness of equation \eref{main equation}. To do this, we first give the definition of the solution.

\begin{definition} We call a predictable $H$-valued stochastic process $\{X_t\}_{t\geq 0}$ defined on the probability space $(\Omega,\mathcal{F},(\mathcal{F}_t)_{t\geq 0},\mathbb{P})$ with the initial value $\xi\in H$ a mild solution of equation (\ref{main equation}) if the following statements are satisfied:

i) $\{X_t\}_{t\geq 0}$ is  $\{\mathcal{F}_t\}_{t\geq 0}$ adapted.

ii) $\PP\left(\int^t_0 |e^{(t-s)A}B(X_s,\mathscr{L}_{X_s})|ds<\infty\right)=1,\quad \forall t\geq 0$.

iii) For any $t\geq 0$, $\{X_t\}_{t\geq 0}$ satisfies
\begin{eqnarray}
X_t=e^{tA}\xi+\int^t_0 e^{(t-s)A}B(X_s,\mathscr{L}_{X_s})ds+\int^t_0 e^{(t-s)A}dL_s,\quad \mathbb{P}-a.s..
\end{eqnarray}

\end{definition}

Next, we present our first main result.
\begin{theorem} \label{main result 1}
Suppose that assumptions \ref{A1}-\ref{A3} hold. Then for any initial value $\xi\in H$ satisfying $\EE|\xi|^{p}<\infty$, equation (\ref{main equation}) has a unique mild solution. Furthermore, if $\EE|\xi|^m<\infty$ for some $m\in [p,\alpha)$, then for any $T>0$, there exists $C_{T,m}>0$ such that
\begin{eqnarray}
\sup_{t\in [0,T]} \left(\mathbb{E}|X_t|^{m}\right)^{1/m}\leq C_{T}\left[1+\left(\EE|\xi|^m\right)^{1/m}\right]. \label{FM}
\end{eqnarray}
\end{theorem}
\begin{proof}
%We will use the fixed point theory to prove the existence and uniqueness  of the solution.
The detailed proof is divided into three steps.

\textbf{Step 1:}  For the fixed $p\in [1,\alpha)$ in assumption \ref{A2}. Let $C([0, T], L^p(\Omega, \PP; H))$ be the Banach space of continuous maps $\{Z_t\}_{t\geq 0}$ from $[0,T]$ to $L^p(\Omega, \PP; H)$ satisfying $\sup_{t\in [0,T]}\EE |Z_t|^p<\infty$. Let $\Lambda_p$ be the closed subspace of $C([0, T], L^p(\Omega, \PP; H))$ consisting of measurable and $\{\mathcal{F}_t\}_{t\geq 0}$ adapted process $\{X_t\}_{t\geq 0}$. Then $\Lambda_p$ is a Banach space with the norm topology given by
$$
\|Z\|_{\Lambda_p}:=\sup_{t\in [0,T]}e^{-\lambda t}\left[\EE |Z_t|^p\right]^{1/p},
$$
where $\lambda>0$ is a large enough constant. We define by $C([0, T], \mathscr{P}_p)$ the complete metric space of continuous functions from $[0,T]$ to $(\mathscr{P}_p, \mathbb{W}_p)$ with the metric:
$$
D_T(\mu, \nu):=\sup_{t\in [0,T]}e^{-\lambda t}\mathbb{W}_p(\mu_t,\nu_t),\quad \mu,\nu \in C([0, T], \mathscr{P}_p).
$$

For a fixed $\mu\in C([0, T], \mathscr{P}_p)$, we first study the well-posedness of the following equation:
\begin{eqnarray}
dX_t=AX_t dt+B(X_t,\mu_t)dt+dL_t,\quad X_0=\xi.\label{E2.5}
\end{eqnarray}
Let us define the operator on $\Lambda_p$ as follows:
\begin{eqnarray*}
(\Phi_{\mu}X)(t)=e^{tA}\xi+\int^t_0 e^{(t-s)A}B(X_s,\mu_s)ds+\int^t_0 e^{(t-s)A}dL_s.
\end{eqnarray*}

On one hand, refer to \cite[(4.12)]{PZ}, the assumption \ref{A3} implies that
\begin{eqnarray}
\sup_{t\geq 0}\EE\left|\int^t_0 e^{(t-s)A}d L_s\right|^p\leq C_{\alpha,p}\left(\sum^{\infty}_{k=1}\frac{\beta^{\alpha}_k}{\lambda_k}\right)^{p/\alpha}<\infty,\quad 0<p<\alpha.\label{LA1}
\end{eqnarray}
Note that by assumption \ref{A2}, it is easy to see that
\begin{eqnarray}
|B(x,\mu)|\leq C\left[1+|x|+\mu(|\cdot|^p)^{1/p}\right],\quad \forall x\in H,\mu\in \mathscr{P}_{p}. \label{LGT}
\end{eqnarray}
Then by \eref{P2}, \eref{LA1}, \eref{LGT} and Minkowski's inequality, it follows for any $t\in [0,T]$
\begin{eqnarray*}
\left[\EE|(\Phi_{\mu}X)(t)|^p\right]^{1/p}\leq\!\!\!\!\!\!\!\!&&\left(\EE|e^{tA}\xi|^p\right)^{1/p}+\int^t_0\left[\EE |e^{(t-s)A}B(X_s,\mu_s)|^p\right]^{1/p}ds+\left[\EE\left|\int^t_0 e^{(t-s)A}dL_s\right|^p\right]^{1/p}\\
\leq\!\!\!\!\!\!\!\!&&\left(\EE|\xi|^p\right)^{1/p}+C\int^t_0 \left(\EE|X_s|^p\right)^{1/p}+\left[\mu_s(|\cdot|^p)\right]^{1/p}ds+\left[\EE\left|\int^t_0 e^{(t-s)A}dL_s\right|^p\right]^{1/p}\\
\leq\!\!\!\!\!\!\!\!&&C_{T,p},
\end{eqnarray*}
which implies that $\|\Phi_{\mu}X\|_{\Lambda_p}<\infty$. It is easy to see that $(\Phi_{\mu}X)(t)$ is $\{\mathcal{F}_t\}_{t\geq 0}$ adapted process if $\{X_t\}_{t\geq 0}$ is adapted, thus
\begin{eqnarray}
\Phi_{\mu}: \Lambda_p\rightarrow \Lambda_p.\label{F2.9}
\end{eqnarray}

On the other hand, for any $X,Y\in \Lambda_p$ with $X_0=Y_0$, we have for any $t\in [0,T]$,
\begin{eqnarray*}
e^{-\lambda t}\left[\EE|(\Phi_{\mu}X)(t)-(\Phi_{\mu}Y)(t)|^p\right]^{1/p}\leq\!\!\!\!\!\!\!\!&&e^{-\lambda t}\left[\EE\left|\int^t_0 e^{(t-s)A}\left[B(X_s,\mu_s)-B(Y_s,\mu_s)\right]ds\right|^p\right]^{1/p}\nonumber\\
\leq\!\!\!\!\!\!\!\!&&e^{-\lambda t}\int^t_0 C\left(\EE|X_s-Y_s|^p\right)^{1/p} ds\nonumber\\
\leq\!\!\!\!\!\!\!\!&&C\int^t_0e^{-\lambda (t-s)} ds \|X-Y\|_{\Lambda_p}\\
\leq\!\!\!\!\!\!\!\!&&\frac{C}{\lambda} \|X-Y\|_{\Lambda_p}.
\end{eqnarray*}
Let us choose $\lambda$ large enough such that $c_0 := \frac{C}{\lambda} <1$, thus it follows
\begin{eqnarray}
\|\Phi_{\mu}X-\Phi_{\mu}Y\|_{\Lambda_p}\leq c_0 \|X-Y\|_{\Lambda_p}. \label{F2.10}
\end{eqnarray}

Hence \eref{F2.9} and \eref{F2.10} yield that $\Phi_{\mu}$ is a contraction map on $\Lambda_p$. As a result, $\Phi_{\mu}$ has a unique fixed point $X_{\mu}$ in $\Lambda_p$, which is the unique solution of equation \eref{E2.5}, i.e.,
\begin{eqnarray}
X_{\mu}(t)=e^{tA}\xi+\int^t_0 e^{(t-s)A}B(X_\mu(s),\mu_s)ds+\int^t_0 e^{(t-s)A}d L_s.
\end{eqnarray}

\textbf{Step 2:}
We define an operator $\Psi$ on $\mu\in C([0, T], \mathscr{P}_p)$ by
$$
\Psi: \mu\rightarrow \mathscr{L}_{X_\mu}.
$$
where $\mathscr{L}_{X_\mu}=\{\mathscr{L}_{X_\mu(t)}, t\in [0,T]\}$ is the law of process $\{X_\mu(t)\}_{t\geq 0}$. Obviously, if $X$ is a mild solution of equation \eref{main equation}. Then its Law $\{\mu_t:=\mathscr{L}_{X_\mu(t)}\}_{t\geq 0}$ is a fixed point of $\Psi$. Thus in order to complete the proof, it is sufficient to show that the operator $\Psi$ has a unique fixed point.

For any $\mu\in C([0, T], \mathscr{P}_p)$, we first prove that $\Psi(\mu)\in C([0, T], \mathscr{P}_p)$. By Step 1, it is easy to see that $\sup_{t\in[0,T]}\EE|X_{\mu}(t)|^p<\infty$. Thus it remains to show that $t\rightarrow \mathscr{L}_{X_\mu(t)}$ is continuous in $(\mathscr{P}_p, \mathbb{W}_p)$.

In fact, note that for any $0\leq t\leq t+h\leq T$,
$$
X_\mu(t+h)-X_\mu(t)=(e^{hA}-I)X_\mu(t)+\int^{t+h}_{t}e^{(t+h-s)A}B(X_{\mu}(s),\mathscr{L}_{X_\mu(s)})ds+\int^{t+h}_t e^{(t+h-s)A}d L_s.
$$

By dominated convergence theorem, we have
\begin{eqnarray}
\lim_{h\rightarrow 0}\left[\EE\left|(e^{hA}-I)X_\mu(t)\right|^p\right]^{1/p}=0.\label{F2.12}
\end{eqnarray}

By \eref{P2} and \eref{LGT}, it is easy to see
\begin{eqnarray}
\lim_{h\rightarrow 0}\left[\EE\left|\int^{t+h}_{t}e^{(t+h-s)A}B(X_{s},\mathscr{L}_{X_\mu(s)})ds\right|^p\right]^{1/p}\!\!\!\!\leq\lim_{h\rightarrow 0} Ch\left(1+\sup_{t\in[0,T]}\left[\EE|X_{\mu}(t)|^p\right]^{1/p}\right)=0.\label{F2.13}
\end{eqnarray}

Define $\{\tilde L_t\}_{t\geq 0}:=\{L_{t+h}-L_{h}\}_{t\geq 0}$, which is also a cylindrical $\alpha$-stable process. Then  refer to \cite[(4.12)]{PZ}, we have
\begin{align*}
\EE\left|\int^{t+h}_t e^{(t+h-s)A}d L_s\right|^p=&\EE\left|\int^{h}_0 e^{(h-s)A}d\tilde L_s\right|^p\\
\leq&\left[\sum^{\infty}_{k=1}\frac{(1-e^{-\alpha \lambda_k h})\beta^{\alpha}_k}{\alpha\lambda_k}\right]^{p/\alpha}.
\end{align*}
Then by dominated convergence theorem and assumption \ref{A3}, we get
\begin{eqnarray}
\lim_{h\rightarrow 0}\left[\EE\left|\int^{t+h}_t e^{(t+h-s)A}d L_s\right|^p\right]^{1/p}\leq \lim_{h\rightarrow 0}C\left[\sum^{\infty}_{k=1}\frac{(1-e^{-\alpha \lambda_k h})\beta^{\alpha}_k}{\alpha\lambda_k}\right]^{1/\alpha}=0.\label{F2.14}
\end{eqnarray}
Thus by \eref{F2.13}-\eref{F2.14}, it is easy to see that
\begin{eqnarray*}
\lim_{h\rightarrow 0}\mathbb{W}_p(\mathscr{L}_{X_\mu(t+h)},\mathscr{L}_{X_\mu(t)})\leq\!\!\!\!\!\!\!\!&&\lim_{h\rightarrow 0}\left[\EE|X_\mu(t+h)-X_\mu(t)|^p\right]^{1/p}\\
\leq\!\!\!\!\!\!\!\!&&\lim_{h\rightarrow 0}\left\{\left[\EE\left|(e^{hA}-I)X_\mu(t)\right|^p\right]^{1/p}+\left[\EE\left|\int^{t+h}_t e^{(t+h-s)A}d L_s\right|^p\right]^{1/p}\right.\\
&&+\left.\left[\EE\left|\int^{t+h}_{t}e^{(t+h-s)A}B(X_{\mu}(s),\mathscr{L}_{X_\mu(s)})ds\right|^p\right]^{1/p}\right\}=0.
\end{eqnarray*}

Next, we shall prove that $\Psi$ is a contraction operator. Let $X_{\mu}$ and $X_{\nu}$  be the corresponding solutions of equation of \eref{E2.5} for $\mu, \nu\in C([0, T], \mathscr{P}_p)$ respectively. Then it follows
\begin{eqnarray*}
\left[\EE|X_{\mu}(t)-X_{\nu}(t)|^p\right]^{1/p}\leq\!\!\!\!\!\!\!\!&&\left[\EE\left|\int^t_0 e^{(t-s)A}\left(B(X_{\mu}(s),\mu_s)-B(X_{\nu}(s),\nu_s)\right)ds\right|^p\right]^{1/p}\nonumber\\
\leq\!\!\!\!\!\!\!\!&&C\int^t_0 \left[\EE|X_{\mu}(s)-X_{\nu}(s)|^p\right]^{1/p} +\mathbb{W}_p(\mu_s,\nu_s)ds,
\end{eqnarray*}
which implies
\begin{eqnarray*}
\sup_{t\in [0,T]}e^{-\lambda t}\left[\EE|X_{\mu}(t)-X_{\nu}(t)|^p\right]^{1/p}\leq\!\!\!\!\!\!\!\!&&\frac{C}{\lambda}\sup_{s\in [0,T]}e^{-\lambda s}\left[\EE|X_{\mu}(s)-X_{\nu}(s)|^p\right]^{1/p}+\frac{C}{\lambda}\sup_{s\in [0,T]}e^{-\lambda s}\mathbb{W}_p(\mu_s,\nu_s).
\end{eqnarray*}
Choose $\lambda>0$ large enough such that $\frac{C}{\lambda}<1/2$, then we get
\begin{eqnarray*}
\sup_{t\in [0,T]}e^{-\lambda t}\left[\EE|X_{\mu}(t)-X_{\nu}(t)|^p\right]^{1/p}\leq\!\!\!\!\!\!\!\!&&\frac{2C}{\lambda}\sup_{t\in [0,T]}e^{-\lambda t}\mathbb{W}_p(\mu_t,\nu_t).
\end{eqnarray*}
Note that $\mathbb{W}_p(\mathscr{L}_{X_\mu(t)},\mathscr{L}_{X_\nu(t)})\leq \left[\EE|X_{\mu}(t)-X_{\nu}(t)|^p\right]^{1/p}$, it follows
\begin{eqnarray*}
D_T(\Psi(\mu),\Psi(\nu))\leq\!\!\!\!\!\!\!\!&&\frac{2C}{\lambda}D_T(\mu,\nu).
\end{eqnarray*}
Hence $\Psi$ is a contraction operator in $C([0, T], \mathscr{P}_p)$.

\textbf{Step 3:} Let $X_t$ be the unique solution of equation (\ref{main equation}). By \cite[(4.12)]{PZ}, it follows for any $m\in [p,\alpha)$ and $t\in [0,T]$, we have
\begin{eqnarray*}
\left(\EE|X_t|^m\right)^{1/m}\leq\!\!\!\!\!\!\!\!&&\left(\EE|e^{tA}\xi|^m\right)^{1/m}+\int^t_0\left[\EE |e^{(t-s)A}B(X_s,\mathscr{L}_{X_s})|^m\right]^{1/m}ds+\left[\EE\left|\int^t_0 e^{(t-s)A}dL_s\right|^m\right]^{1/m}\\
\leq\!\!\!\!\!\!\!\!&&\left(\EE|\xi|^m\right)^{1/m}+\int^t_0 C\left(\EE|X_s|^m\right)^{1/m}+C\left(\EE|X_s|^p\right)^{1/p}ds+\left[\EE\left|\int^t_0 e^{(t-s)A}dL_s\right|^m\right]^{1/m}\\
\leq\!\!\!\!\!\!\!\!&&C\left[1+\left(\EE|\xi|^m\right)^{1/m}\right]+\int^t_0 C\left(\EE|X_s|^m\right)^{1/m}ds.
\end{eqnarray*}
Then by Gronwall's inequality, it is easy to see
\begin{eqnarray*}
\left(\EE|X_t|^m\right)^{1/m}\leq Ce^{Ct}\left[1+\left(\EE|\xi|^m\right)^{1/m}\right].
\end{eqnarray*}
The proof is complete.
\end{proof}

\section{Averaging principle for multiscale Mckean-Vlasov SPDEs driven by $\alpha$-stable processes}

In this section, we further study a class of multiscale Mckean-Vlasov SPDEs driven by $\alpha$-stable processes. Under some proper assumptions, we will prove the strong averaging principle holds with some order $r>0$. The main technique is based on the classical  Khasminskii's time discretization method.

\subsection{Assumptions and main result}

We recall the multiscale Mckean-Vlasov SPDEs driven by $\alpha$-stable processes in the Hilbert space $H$:
\begin{equation}\left\{\begin{array}{l}\label{main equation 1}
\displaystyle
dX^{\vare}_t=\left[AX^{\vare}_t+F(X^{\vare}_t, \mathscr{L}_{X^{\vare}_t}, Y^{\vare}_t)\right]dt+dL_t,\quad X^{\vare}_0=\xi\in H,\\
dY^{\vare}_t=\frac{1}{\vare}[AY^{\vare}_t+G(X^{\vare}_t, \mathscr{L}_{X^{\vare}_t},Y^{\vare}_t)]dt+\frac{1}{\vare^{1/\alpha}}dZ_t,\quad Y^{\vare}_0=\eta\in H,\end{array}\right.
\end{equation}
where operator $A$ satisfies assumption \ref{A1}. $\xi,\eta$ are two  $H$-valued random variable. $\{L_t\}_{t\geq 0}$ and $\{Z_t\}_{t\geq 0}$ are two mutually independent cylindrical $\alpha$-stable processes with $\alpha\in(1,2)$ on a probability space $(\Omega,\mathcal{F},\mathbb{P})$ with natural filtration $\{\mathcal{F}_t:=\sigma(\xi,\eta,L_s,Z_s,s\leq t)\}_{t\geq 0}$ , i.e.,
$$
L_t=\sum^{\infty}_{k=1}\beta_{k}L^{k}_{t}e_k,\quad Z_t=\sum^{\infty}_{k=1}\gamma_{k}Z^{k}_{t}e_k,\quad t\geq 0,
$$
where $\{\beta_k\}_{k\in \mathbb{N}_{+}}$ and $\{\gamma_k\}_{k\in\mathbb{N}_{+}}$ are two given sequence of positive numbers, $\{L^k_t\}_{k\in \mathbb{N}_{+}}$ and $\{Z^k_t\}_{k\in \mathbb{N}_{+}}$ are two sequences of independent one dimensional symmetric $\alpha$-stable processes satisfying for any $k\in \mathbb{N}_{+}$ and $t\geq0$,
$$\mathbb{E}[e^{i L^k_{t}h}]=\mathbb{E}[e^{i Z^k_{t}h}]=e^{-t|h|^{\alpha}}, \quad h\in \mathbb{R}.$$

Now, we suppose the following conditions hold throughout this section:
\begin{conditionB}\label{B1}
Suppose that there exists a constant $p\in [1,\alpha)$ such that
$$
F(\cdot,\cdot,\cdot):H\times \mathscr{P}_p\times H\rightarrow H,
$$
$$
G(\cdot,\cdot,\cdot):H\times \mathscr{P}_p\times H\rightarrow H.
$$
Furthermore, there exist constants $C, L_{G}>0$ such that for any $x_i,y_i\in H$ and $\mu_i\in \mathscr{P}_p$, $i=1,2$,
\begin{eqnarray*}
&&\left|F(x_1,\mu_1,y_1)-F(x_2,\mu_2,y_2)\right|\leq C\left[|x_1-x_2|+\mathbb{W}_p(\mu_1,\mu_2)+|y_1-y_2|\right],\label{LipF}\\
&&\left|G(x_1,\mu_1,y_2)-G(x_2,\mu_2,y_2)\right|\leq C\left[|x_1-x_2|+\mathbb{W}_p(\mu_1,\mu_2)\right]+L_{G}|y_1-y_2|.\label{LipG}
\end{eqnarray*}
\end{conditionB}

\begin{conditionB}\label{B2} There exists $\theta\in (0,2/\alpha]$ such that
$\sum^{\infty}_{k=1}\frac{\beta^{\alpha}_k}{\lambda^{1-\alpha\theta/2}_k}<\infty$ and $\sum^{\infty}_{k=1}\frac{\gamma^{\alpha}_k}{\lambda_k}<\infty$.
\end{conditionB}

\begin{conditionB}\label{B3}
Suppose that $\lambda_{1}-L_{G}>0$ and there exists $C>0$ such that
\begin{eqnarray}
\sup_{x,y\in H}|F(x,\mu,y)|\leq C\left[1+\left(\mu(|\cdot|^p)\right)^{1/p}\right].\label{BDC}
\end{eqnarray}
\end{conditionB}

\vspace{0.2cm}

The following are some comments on the assumptions above:
\begin{remark}\label{Re1} Suppose that assumptions \ref{A1}, \ref{B1} and \ref{B2} holds. By Theorem \ref{main result 1}, for any given $\varepsilon>0$ and initial value $(\xi, \eta)\in H\times H$ satisfying $\EE|\xi|^{p}<\infty$ and $\EE|\eta|^{p}<\infty$, equation \eref{main equation 1} admits a unique mild solution $(X^{\varepsilon}_t, Y^{\varepsilon}_t)$, i.e., $\PP$-a.s.,
\begin{equation}\left\{\begin{array}{l}\label{A mild solution}
\displaystyle
X^{\varepsilon}_t=e^{tA}\xi+\int^t_0e^{(t-s)A}F(X^{\varepsilon}_s, \mathscr{L}_{X^{\vare}_s},Y^{\varepsilon}_s)ds+\int^t_0 e^{(t-s)A}dL_s,\\
Y^{\varepsilon}_t=e^{tA/\varepsilon}\eta+\frac{1}{\varepsilon}\int^t_0e^{(t-s)A/\varepsilon}G(X^{\varepsilon}_s, \mathscr{L}_{X^{\vare}_s},Y^{\varepsilon}_s)ds
+\frac{1}{\vare^{1/\alpha}}\int^t_0 e^{(t-s)A/\varepsilon}dZ_s.
\end{array}\right.
\end{equation}
\end{remark}
\begin{remark}
The condition $\lambda_{1}-L_{G}>0$ in assumption \ref{B3} is called the strong dissipative condition, which is used to prove the existence and uniqueness of the invariant measures and the exponential ergodicity of the transition semigroup of the frozen equation.
\end{remark}
\begin{remark}
%For the reason that the solution $(X^{\varepsilon}_t, Y^{\varepsilon}_t)$ does not has finite second moment, thus we assume the condition \eref{BDC}.
The method used here is the classical Khasminskii's time discretization, which highly depends on the square calculation in the proof, hence the finite second moment of the solution $X^{\vare}_t$ is required usually. But the solution $X^{\vare}_t$ for system (1.1) only has finite $p$-th moment $(0< p < \alpha)$, thus the  condition \eref{BDC} in assumption \ref{B3} is used to weaken the required finite second moment to finite first moment. However by the technique of Poisson equation (see e.g. \cite{RSX,SXX}), the condition \eref{BDC} could be removed, which is left to our further work.
\end{remark}
\begin{remark}\label{Re2}
Refer to \cite[Lemma 4.1]{PSXZ}, if $\sum^{\infty}_{k=1}\frac{\beta^{\alpha}_k}{\lambda^{1-\alpha\theta/2}_k}<\infty$ holds for some $\theta\geq 0$ then for any  $0<m<\alpha$, we have
\begin{eqnarray}
\sup_{t\geq 0}\EE\left\|\int^t_0 e^{(t-s)A}d L_s\right\|^m_{\theta}\leq C_{\alpha,p}\left(\sum^{\infty}_{k=1}\frac{\beta^{\alpha}_k}{\lambda^{1-\alpha\theta/2}_k}\right)^{m/\alpha}.\label{LA}
\end{eqnarray}
%Similar, if $\sum_{k\in \mathbb{N}_{+}}\frac{\gamma^{\alpha}_k}{\lambda^{1-\alpha\theta/2}_k}<\infty$ holds for some $\theta\geq 0$ then for any  $0<p<\alpha$, we have
%\begin{eqnarray}
%\sup_{t\geq 0}\EE\left\|\int^t_0 e^{(t-s)A}d Z_s\right\|^p_{\theta}\leq C_{\alpha,p}\left(\sum_{k\in\mathbb{N}_{+}}\frac{\gamma^{\alpha}_k}{\lambda^{1-\alpha\theta/2}_k}\right)^{p/\alpha}.\label{ZA}
%\end{eqnarray}
\end{remark}

Now, we state our second result in this paper, whose proof is left in subsection 3.3 below.
\begin{theorem} \label{main result 2}
Suppose that assumptions \ref{A1} and \ref{B1}-\ref{B3} hold. Then for any $T>0$, $\xi,\eta\in H$ satisfying $\EE|\xi|^m<\infty$ for some $m\in [p,\alpha)$ and $\EE|\eta|^p<\infty$, then we have
\begin{align}
\left[\mathbb{E}\left(\sup_{t\in [0,T]}|X_{t}^{\vare}-\bar{X}_{t}|^{m}\right)\right]^{1/m}\leq C_{T}\left[1+\left(\EE|\xi|^m\right)^{1/m}+\left(\EE|\eta|^{p}\right)^{1/p}\right]\vare^{\frac{\theta}{2(1+\theta)}}. \label{ST}
\end{align}
\end{theorem}
\begin{remark}
In contrast to the model in \cite{BYY}, the coefficients $F$ and $G$ in equation \eref{main equation 1} all depend on the distribution. Meanwhile, the main result \eref{ST} implies that the strong convergence order is $\frac{\theta}{2(1+\theta)}$, which depends on the regular assumption on the noise in the slow component.
\end{remark}

\subsection {Some a priori estimates}
In this subsection, we will prove some a priori estimates of the solutions $(X_{t}^{\varepsilon}, Y_{t}^{\vare})$ and an auxiliary process $\hat Y_{t}^{\vare}$. Note that we always assume that the initial value $\xi,\eta$ satisfying $\EE|\xi|^m<\infty$ for some $m\in [p,\alpha)$ and $\EE|\eta|^p<\infty$.

\begin{lemma} \label{PMY}
For any $T>0$, there exists a constant $C_{T}>0$ such that
\begin{align}
\sup_{\vare\in(0,1),t\in [0,T] }\left(\mathbb{E}|X_{t}^{\vare}|^{m}\right)^{1/m}\leq C_{T}\left[1+\left(\EE|\xi|^{m}\right)^{1/m}\right]\label{EX}
\end{align}
and
\begin{align}
\sup_{\vare\in(0,1),t\in [0,T]}\left(\mathbb{E}|Y_{t}^{\varepsilon}|^{p}\right)^{1/p}\leq C_{T}\left[1+\left(\EE|\xi|^p\right)^{1/p}+\left(\EE|\eta|^{p}\right)^{1/p}\right].\label{EY}
\end{align}
\end{lemma}
\begin{proof}
%Let $\mathcal{H}:=H\times H$ be the product Hilbert space. Rewrite the system \eref{Main equation} for $Z^{\varepsilon}_t=(X^{\varepsilon}_t,Y^{\varepsilon}_t)$ as
%\begin{eqnarray*}
%dZ^{\varepsilon}_t=\tilde{A}Z^{\varepsilon}_tdt+G^{\vare}(Z^{\varepsilon}_t)dt+dW_t,\quad Z^{\varepsilon}_0=(x,y)\in \mathcal{H},
%\end{eqnarray*}
%where $W_t:=(W_t^{1},W_t^{2})$  is a $\mathcal{H}$-valued cylindrical-Wiener process, $Q$ is a bounded operator in  $\mathcal{H}$, which is denoted by $Qz=(Q_1x, Q_2y)$, for $z=(x,y)\in \mathcal{H}$, and
%\begin{eqnarray*}
%&&\tilde{A}Z^{\varepsilon}_t=\left(AX^{\varepsilon}_t,\frac{1}{\varepsilon}AY^{\varepsilon}_t\right),
%\\&&G^{\vare}(Z^{\varepsilon}_t)=\left(B(X^{\varepsilon}_t,Y^{\varepsilon}_t),\frac{1}{\varepsilon}F(X^{\varepsilon}_t,Y^{\varepsilon}_t)\right).
%\end{eqnarray*}
%It is easy to see that $G^{\vare}$ is Lipschitz continuous  in $\mathcal{H}$, i.e.,
%$$
%\|G^{\vare}(z_1)-G^{\vare}(z_2)\|_{\mathcal{H}}\leq C_{\vare}\|z_1-z_2\|^{\alpha\wedge\beta\wedge\gamma}_{\mathcal{H}},\quad z_1,z_2\in \mathcal{H}.
%$$
%Then under the assumptions \ref{A1}-\ref{A5}, the existence and uniqueness of strong solution in the mild sense for system \eqref{Main equation} follows by \cite[Theorem 7]{DF}.
%Note that we consider the time-independent SDE, so the condition (6) in \cite{DF} is replaced by condition \eref{A50} here, in fact these two conditions are the same essentially.
Recall that
$$
X^{\varepsilon}_t=e^{tA}\xi+\int^t_0 e^{(t-s)A}F(X^{\varepsilon}_s, \mathscr{L}_{X^{\vare}_s},Y^{\varepsilon}_s)ds+\int^t_0 e^{(t-s)A}dL_s.
$$
Then by \eref{BDC}, \eref{LA} and Minkowski's inequality, we have
\begin{eqnarray*}
\sup_{t\in[0,T]}\left(\EE|X_{t}^{\varepsilon}|^{m}\right)^{1/m}\leq\!\!\!\!\!\!\!\!&&\left(\EE|\xi|^m\right)^{1/m}+C\int^T_0 \left(\EE|X_{t}^{\varepsilon}|^{p}\right)^{1/p}dt+\sup_{t\in[0, T]}\left[\EE\left|\int^t_0 e^{(t-s)A}dL_s\right|^{m}\right]^{1/m}\nonumber\\
\leq\!\!\!\!\!\!\!\!&&C\left[1+\left(\EE|\xi|^m\right)^{1/m}\right]+C\int^T_0 \left(\EE|X_{t}^{\varepsilon}|^{m}\right)^{1/m}dt+C\left(\sum^{\infty}_{k=1}\frac{\beta^{\alpha}_k}{\lambda^{1-\alpha\theta/2}_k}\right)^{m/\alpha}.
\end{eqnarray*}
Then the Gronwall's inequality implies that \eref{EX} holds.

Now, we proceed to show estimate (\ref{EY}). Recall that
$$
Y^{\varepsilon}_t=e^{tA/\varepsilon}\eta+\frac{1}{\varepsilon}\int^t_0e^{(t-s)A/\varepsilon}G(X^{\varepsilon}_s, \mathscr{L}_{X^{\vare}_s},Y^{\varepsilon}_s)ds
+\frac{1}{\vare^{1/\alpha}}\int^t_0 e^{(t-s)A/\varepsilon}dZ_s.
$$
Note that by assumption \ref{B1}, it is easy to see that
\begin{eqnarray}
|G(x,\mu,y)|\leq C\left[1+|x|+\mu(|\cdot|^p)^{1/p}\right]+L_G|y|,\quad \forall x,y\in H,\mu\in \mathscr{P}_{p}. \label{LFT}
\end{eqnarray}
Then by \eref{P2} and \eref{LFT}, we have for any $t\geq0$,
\begin{align*}
|Y^{\varepsilon}_t|\leq\!\!|\eta|+\frac{1}{\varepsilon}\int^t_0\!\!e^{-\lambda_1(t-s)/\varepsilon}(C+C|X^{\varepsilon}_s|+C\left[\EE|X^{\varepsilon}_s|^p\right]^{1/p}+L_{G}|Y^{\varepsilon}_s|)ds
+\left|\frac{1}{\vare^{1/\alpha}}\int^t_0 e^{(t-s)A/\varepsilon}dZ_s\right|.
\end{align*}

Define $\tilde Z_t:=\frac{1}{\vare^{1/ \alpha}}Z_{t\vare}$, which is also a cylindrical $\alpha$-stable process. Then by \cite[(4.12)]{PZ},
\begin{align*}
\EE\left|\frac{1}{\vare^{1/\alpha}}\int^t_0 e^{(t-s)A/\varepsilon}dZ_s\right|^p=&\EE\left|\int^{t/\vare}_0 e^{(t/\vare-s)A}d\tilde Z_s\right|^p\\
\leq&C\left(\sum^{\infty}_{k=1}\gamma^{\alpha}_k \frac{1-e^{-\alpha \lambda_k t/\vare}}{\alpha \lambda_k} \right)^{p/\alpha}\\
\leq&C\left(\sum^{\infty}_{k=1}\frac{\gamma^{\alpha}_k }{\alpha \lambda_k} \right)^{p/\alpha},
\end{align*}
which together with \eref{EX}, by Minkowski's inequality, we have for any $t\leq T$,
\begin{align*}
\left(\EE|Y^{\varepsilon}_t|^p\right)^{1/p}\leq&\left(\EE|\eta|^p\right)^{1/p}+\frac{C}{\varepsilon}\int^t_0e^{-\lambda_1(t-s)/\varepsilon}ds+\frac{C}{\varepsilon}\int^t_0e^{-\lambda_1(t-s)/\varepsilon}\left(\EE|X^{\varepsilon}_s|^p\right)^{1/p}ds\\
&+\frac{L_G}{\varepsilon}\int^t_0e^{-\lambda_1(t-s)/\varepsilon}\left(\EE|Y^{\varepsilon}_s|^p\right)^{1/p}ds
+\left[\EE\left|\frac{1}{\vare^{1/\alpha}}\int^t_0 e^{(t-s)A/\varepsilon}dZ_s\right|^p\right]^{1/p}\\
\leq& C_T\left[1+\left(\EE|\xi|^p\right)^{1/p}+\left(\EE|\eta|^p\right)^{1/p}\right]+\frac{L_{G}}{\lambda_1}\sup_{0\leq t\leq T}\left(\EE|Y^{\varepsilon}_t|^p\right)^{1/p}.
\end{align*}
Hence by the condition $L_G<\lambda_1$ in assumption \ref{B3}, it is easy to see that \eref{EY} holds. The proof is complete.
\end{proof}

\begin{remark}
Note that if without the condition \eref{BDC}. Then by \eref{LGT}, we can obtain following a prior estimate:
$$
\sup_{\vare\in(0,1),t\in [0,T] }\left[\left(\mathbb{E}|X_{t}^{\vare}|^{m}\right)^{1/m}+\left(\mathbb{E}|Y_{t}^{\varepsilon}|^{m}\right)^{1/m}\right]\leq C_{T}\left[1+\left(\EE|\xi|^{m}\right)^{1/m}+\left(\EE|\eta|^{m}\right)^{1/m}\right].
$$
\end{remark}

%Usually, the H\"{o}lder continuity of $X_{t}^{\varepsilon}$ in time plays an important role in the method of time discretization (see \cite[Proposition 4.4]{C1}, \cite[Lemma 3,4]{DSXZ} and \cite[Proposition 9]{GP}), then the initial value $x\in H^{\theta}$ will be assumed for some $\theta>0$. However inspired from \cite{LRSX2}, studying the H\"{o}lder continuity can be replaced by studying the integral of the time increment of $X_{t}^{\varepsilon}$, which is weaker than the H\"{o}lder continuity but  enough for our purpose, and it only needs initial value $x\in H$ for advantage.

\begin{lemma} \label{L3.9}
For any $x,y\in H$ and $T>0$, there exists a constant $C_{T}>0$ such that for any $\varepsilon\in(0,1)$, we have
\begin{align}
\int^{T}_0\left[\mathbb{E}|X^{\varepsilon}_t-X^{\varepsilon}_{t(\delta)}|^m\right]^{1/m}dt\leq C_T\delta^{\frac{\theta }{2}}\left[1+\left(\EE|\xi|^{m}\right)^{1/m}\right],\label{COX}
\end{align}
where $\theta$ is the one in assumption \ref{B2} and $t(\delta):=[\frac{t}{\delta}]\delta$ with $[\frac{t}{\delta}]$ is the integer part of $\frac{t}{\delta}$.
\end{lemma}

\begin{proof}
By (\ref{EX}) and Minkowski's inequality, it is easy to check that
\begin{eqnarray}
\int^{T}_0\left[\mathbb{E}|X_{t}^{\varepsilon}-X_{t(\delta)}^{\varepsilon}|^m\right]^{1/m}dt=\!\!\!\!\!\!\!\!&&\int^{\delta}_0\left(\mathbb{E}|X_{t}^{\varepsilon}-\xi|^m\right)^{1/m}dt
+\int^{T}_{\delta}\left[\mathbb{E}|X_{t}^{\varepsilon}-X_{t(\delta)}^{\varepsilon}|^m\right]^{1/m}dt\nonumber\\
\leq\!\!\!\!\!\!\!\!&&C_{T}\delta\left[1+\left(\EE|\xi|^m\right)^{1/m}\right]+\int^{T}_{\delta}\left(\mathbb{E}|X_{t}^{\varepsilon}-X_{t-\delta}^{\varepsilon}|^m\right)^{1/m}dt\nonumber\\
&&+\int^{T}_{\delta}\left(\mathbb{E}|X_{t(\delta)}^{\varepsilon}-X_{t-\delta}^{\varepsilon}|^m\right)^{1/m}dt.\label{FFX1}
\end{eqnarray}

Recall that the mild solution $X^{\vare}_t$ in \eref{A mild solution}. Then by \eref{P3}, \eref{BDC} and Remark \ref{Re2}, it is easy to see that for any $\theta\in (0,2/\alpha]$, $t\in(0,T]$ and $m\in[p,\alpha)$,
\begin{eqnarray}\label{G1}
\left(\mathbb{E}\|X_{t}^{\varepsilon}\|_{\theta}^{m}\right)^{1/m}\leq\!\!\!\!\!\!\!\!&&\left(\EE\|e^{tA}\xi\|^m_{\theta}\right)^{1/m}+\left[\EE\left|\int^t_0(-A)^{\theta/2}e^{(t-s)A}F(X^{\varepsilon}_s, \mathscr{L}_{X^{\vare}_s}, Y^{\varepsilon}_s)ds\right|^m\right]^{1/m}\nonumber\\
&&+\left[\EE\left\|\int^t_0 e^{(t-s)A}dL_s\right\|^m_{\theta}\right]^{1/m}\nonumber\\
\leq\!\!\!\!\!\!\!\!&&C t^{-\theta/2}\left(\EE|\xi|^{m}\right)^{1/m}+C\int^t_0(t-s)^{-\frac{\theta}{2}}\left(\EE|X^{\varepsilon}_s|^p\right)^{1/p}ds\nonumber\\
&&+C\left( \sum^{\infty}_{k=1}\frac{\beta^{\alpha}_k}{\lambda^{1-\alpha\theta/2}_k}\right)^{1/\alpha}\nonumber\\
\leq\!\!\!\!\!\!\!\!&&C t^{-\theta/2}\left(\EE|\xi|^{m}\right)^{1/m}+C_{T}.
\end{eqnarray}

Note that
$$X_{t}^{\varepsilon}-X_{t-\delta}^{\varepsilon}=(e^{A\delta}-I)X_{t-\delta}^{\varepsilon}+\int_{t-\delta}^{t}e^{(t-s)A}F(X^{\varepsilon}_s,  \mathscr{L}_{X^{\vare}_s},Y^{\varepsilon}_s)ds+\int_{t-\delta}^{t}e^{(t-s)A}dL_s.$$
Then by \eref{P4} and Minkowski's inequality, we have
\begin{eqnarray}  \label{REGX1}
\int^{T}_\delta\left[\mathbb{E}|(e^{A\delta}-I)X_{t-\delta}^{\varepsilon}|^m\right]^{1/m} dt
%\leq\!\!\!\!\!\!\!\!&&C\delta^{\frac{\theta}{2}}\mathbb{E}\left[\int^{T\wedge \tau^{\varepsilon}_R}_\delta\|X^{\varepsilon}_{t-\delta}\|_{\theta}dt\right]\nonumber\\
\leq\!\!\!\!\!\!\!\!&&C\delta^{\frac{\theta}{2}}\int^{T}_\delta\left(\mathbb{E}\|X^{\varepsilon}_{t-\delta}\|^m_{\theta}  \right)^{1/m}dt\nonumber\\
\leq\!\!\!\!\!\!\!\!&&C\delta^{\frac{\theta}{2}}\left[\int^{T}_\delta (t-\delta)^{-\theta/2}dt\left(\EE|\xi|^{m}\right)^{1/m}+C_T\right]\nonumber\\
\leq\!\!\!\!\!\!\!\!&&C_T\delta^{\frac{\theta }{2}}\left[1+\left(\EE|\xi|^{m}\right)^{1/m}\right].
\end{eqnarray}
By \eref{BDC}, it follows
\begin{eqnarray} \label{REGX3}
\int^{T}_\delta \left[\mathbb{E}\left|\int_{t-\delta}^{t}e^{(t-s)A}F(X^{\varepsilon}_s,\mathscr{L}_{X^{\vare}_s},  Y^{\varepsilon}_s)ds\right|^m\right]^{1/m}dt
\leq\!\!\!\!\!\!\!\!&&C_{T}\delta\left[1+\left(\EE|\xi|^{m}\right)^{1/m}\right].
\end{eqnarray}
By \eref{LA} and assumption \ref{B2}, we have
\begin{eqnarray} \label{REGX4}
\int^{T}_\delta \left[\mathbb{E}\left|\int_{t-\delta}^{t}e^{(t-s)A}dL_s \right|^m\right]^{1/m}dt
\leq\!\!\!\!\!\!\!\!&&C\int^{T}_\delta\left[ \sum^{\infty}_{k=1}\frac{\beta^{\alpha}_k(1-e^{-\lambda_k \delta})}{\lambda_i}\right]^{1/\alpha}dt\nonumber\\
\leq\!\!\!\!\!\!\!\!&&C\delta^{\frac{\theta}{2}}\int^{T}_\delta\left(\sum^{\infty}_{k=1}\frac{\beta^{\alpha}_k}{\lambda^{1-\alpha\theta/2}_k}\right)^{1/\alpha}dt
\leq C_T\delta^{\frac{\theta}{2}},
\end{eqnarray}
where we use the fact that $1-e^{-x}\leq Cx^{\alpha\theta /2}$ for any $x>0$.

Combining \eqref{REGX1}-\eqref{REGX4}, we obtain
\begin{eqnarray}\label{FFX2}
\int^{T}_{\delta}\left(\mathbb{E}|X_{t}^{\varepsilon}-X_{t-\delta}^{\varepsilon}|^m\right)^{1/m}dt\leq\!\!\!\!\!\!\!\!&&C_T\delta^{\frac{\theta }{2}}\left[1+\left(\EE|\xi|^{m}\right)^{1/m}\right].
\end{eqnarray}
Similar as the argument above, we also have
\begin{eqnarray}\label{FFX3}
\int^{T}_{\delta}\left[\mathbb{E}|X_{t(\delta)}^{\varepsilon}-X_{t-\delta}^{\varepsilon}|^m\right]^{1/m}dt\leq\!\!\!\!\!\!\!\!&&C_T\delta^{\frac{\theta }{2}}\left[1+\left(\EE|\xi|^{m}\right)^{1/m}\right].
\end{eqnarray}

Finally, \eref{FFX1}, (\ref{FFX2}) and (\ref{FFX3}) imply \eref{COX} holds. The proof is complete.
\end{proof}
%\begin{remark}\label{Re3}
%Usually, the H\"{o}lder continuity of $X_{t}^{\varepsilon}$ with respect to time plays an important role in the method of time discretization (see e.g. \cite[Proposition 4.4]{C1}, \cite[Lemma 3,4]{DSXZ} and \cite[Proposition 9]{GP}), then the initial value $x\in H^{\theta}$ will be assumed for some $\theta>0$. However, the above Lemma \ref{COX} is an estimate of the integral of the time increment of $X_{t}^{\varepsilon}$, which is weaker than the H\"{o}lder continuity but strong enough for our purpose, and it only needs initial value $x\in H$ for advantage. The idea is inspired from \cite[Lemma 3.2]{LRSX2}, where the SPDEs with the locally monotone coefficients driven by Wiener noise were considered. But here only exists the mild solution, which is unlike the strong solution in \cite{LRSX2}, we have to use the approach based on the smoothing properties of the semigroup $e^{At}$ instead of It\^{o} formula, which is different from the proof as in  \cite[Lemma 3.2]{LRSX2}. Meanwhile unlike the solution $X^{\varepsilon}_t$ has finite moment of any order in \cite{LRSX2}, here $X^{\varepsilon}_t$ has no second moment due to the $\alpha$-stable noise, so the technique of stopping time is helpful to deal with the Burger term.
%\end{remark}

\vspace{0.2cm}
Inspired from the idea introduced by Khasminskii in \cite{K1} ,
we construct an auxiliary process
$\hat{Y}_{t}^{\varepsilon}\in H $, i.e., we split $[0,T]$ into some subintervals of size $\delta>0$, where $\delta>0$ depends on $\vare$ and will be chosen later.
With the initial value $\hat{Y}_{0}^{\varepsilon}=Y^{\varepsilon}_{0}=\xi$, we construct the process $\hat{Y}_{t}^{\varepsilon}$ on each time interval $[l\delta,(l+1)\delta\wedge T]$, $l \in \mathbb{N}$,
$$
d\hat{Y}_{t}^{\varepsilon}=\frac{1}{\varepsilon}\left[A\hat{Y}_{t}^{\varepsilon}+G(X_{l\delta}^{\varepsilon},\mathscr{L}_{X^{\vare}_{l\delta}}, \hat{Y}_{t}^{\varepsilon})\right]dt+\frac{1}{\varepsilon^{1/\alpha}}dZ_t, \quad t\in[l\delta,(l+1)\delta\wedge T],
$$
i.e.,
$$
d\hat{Y}_{t}^{\varepsilon}=\frac{1}{\varepsilon}\left[A\hat{Y}_{t}^{\varepsilon}+G(X_{t(\delta)}^{\varepsilon},\mathscr{L}_{X^{\vare}_{t(\delta)}}, \hat{Y}_{t}^{\varepsilon})\right]dt+\frac{1}{\varepsilon^{1/\alpha}}dZ_t,
$$
which satisfies for any $t\in [0,T]$
\begin{align} \label{AuxiliaryPro Y 01}
\hat{Y}_{t}^{\varepsilon}=e^{tA/\vare}\xi+\frac{1}{\varepsilon}\int_{0}^{t}e^{(t-s)A/\vare}G(X_{s(\delta)}^{\varepsilon},\mathscr{L}_{X^{\vare}_{s(\delta)}},\hat{Y}_{s}^{\varepsilon})ds+\frac{1}{\varepsilon^{1/\alpha}}\int_{0}^{t}e^{(t-s)A/\vare}dZ_s.
\end{align}
By following the same argument as in the proof of \eref{EY}, we have
\begin{align}
\sup_{\vare\in(0,1),t\in [0,T]}\left(\mathbb{E}|\hat{Y}_{t}^{\varepsilon}|^{p}\right)^{1/p}\leq C_{T}\left[1+\left(\EE|\xi|^p\right)^{1/p}+\left(\EE|\eta|^{p}\right)^{1/p}\right].\label{EAY}
\end{align}

\begin{lemma} \label{DEY}
For any $T>0$, there exists a constant $C_{T}>0$ such that
$$
\int^{T}_0\left(\EE|Y_{t}^{\varepsilon}-\hat Y_{t}^{\varepsilon}|^m\right)^{1/m} dt
\leq C_T\delta^{\frac{\theta }{2}}\left[1+\left(\EE|\xi|^{m}\right)^{1/m}\right].
$$
\end{lemma}
\begin{proof}
By the construction of $Y_{t}^{\varepsilon}$ and $\hat{Y}_{t}^{\varepsilon}$, we have
\begin{align*}
Y_{t}^{\varepsilon}-\hat Y_{t}^{\varepsilon}=\frac{1}{\vare}\int^t_0 e^{(t-s)A/\vare}\left[G(X_{s}^{\varepsilon},\mathscr{L}_{X^{\vare}_s},Y_{s}^{\varepsilon})-G(X_{s(\delta)}^{\varepsilon},\mathscr{L}_{X^{\vare}_{s(\delta)}}\hat{Y}_{s}^{\varepsilon})\right]ds.
\end{align*}
Then for any $t>0$,
\begin{align*}
|Y_{t}^{\varepsilon}-\hat Y_{t}^{\varepsilon}|\leq&\frac{1}{\vare}\int^t_0 e^{-\lambda_1(t-s)/\vare}\left[C|X_{s}^{\varepsilon}-X_{s(\delta)}^{\varepsilon}|+C\left(\EE|X_{s}^{\varepsilon}-X_{s(\delta)}^{\varepsilon}|^p\right)^{1/p}\right]ds\nonumber\\
&+\frac{1}{\vare}\int^t_0 e^{-\lambda_1(t-s)/\vare}L_G|Y_{s}^{\varepsilon}-\hat Y_{s}^{\varepsilon}|ds.
\end{align*}

By Fubini's theorem, we have
\begin{align*}
\int^{T}_0|Y_{t}^{\varepsilon}-\hat Y_{t}^{\varepsilon}|dt\leq&\frac{1}{\vare}\int^{T}_0\int^t_0 e^{-\lambda_1(t-s)/\vare}\left[C|X_{s}^{\varepsilon}-X_{s(\delta)}^{\varepsilon}|+\left(\EE|X_{s}^{\varepsilon}-X_{s(\delta)}^{\varepsilon}|^p\right)^{1/p}\right]dsdt\nonumber\\
&+\frac{1}{\vare}\int^{T}_0\int^t_0 e^{-\lambda_1(t-s)/\vare}L_G|Y_{s}^{\varepsilon}-\hat Y_{s}^{\varepsilon}|dsdt\\
=&\frac{C}{\vare}\int^{T}_0\left(\int^{T}_s e^{-\lambda_1(t-s)/\vare}dt\right)|X_{s}^{\varepsilon}-X_{s(\delta)}^{\varepsilon}|ds\nonumber\\
&+\frac{C}{\vare}\int^{T}_0\left(\int^{T}_s e^{-\lambda_1(t-s)/\vare}dt\right)\left(\EE|X_{s}^{\varepsilon}-X_{s(\delta)}^{\varepsilon}|^p\right)^{1/p}ds\nonumber\\
&+\frac{L_G}{\vare}\int^{T}_0\left(\int^{T}_s e^{-\lambda_1(t-s)/\vare}dt\right)|Y_{s}^{\varepsilon}-\hat Y_{s}^{\varepsilon}|ds\\
\leq&\frac{C}{\lambda_1}\int^{T}_0|X_{s}^{\varepsilon}-X_{s(\delta)}^{\varepsilon}|ds+\frac{C}{\lambda_1}\int^{T}_0\left(\EE|X_{s}^{\varepsilon}-X_{s(\delta)}^{\varepsilon}|^p\right)^{1/p}ds+\frac{L_G}{\lambda_1}\int^{T}_0|Y_{s}^{\varepsilon}-\hat Y_{s}^{\varepsilon}|ds.
\end{align*}
By  Minkowski's inequality, Lemma \ref{COX} and $L_G<\lambda_1$, it is easy to see
\begin{align*}
\int^{T}_0\left(\EE|Y_{t}^{\varepsilon}-\hat Y_{t}^{\varepsilon}|^m\right)^{1/m} dt
\leq&C_T\delta^{\frac{\theta }{2}}\left[1+\left(\EE|\xi|^{m}\right)^{1/m}\right].
\end{align*}
The proof is complete.
\end{proof}

\subsection{The frozen and averaged equation}

%The following lemma is used to prove the existence and uniqueness of the solution of corresponding averaged equation, we state it ahead.
%\begin{lemma} \label{L3.17} For any $x_1, x_2\in H$, $y\in H$,  we have
%\begin{align*}
%\sup_{t\geq 0}\|Y^{x_1,y}_t-Y^{x_2,y}_t\|^2\leq (\lambda_1-L_g)^{-1}\|x_1-x_2\|^2.
%\end{align*}
%\end{lemma}
%\begin{proof}
%Note that
%\begin{align*}
%d(Y^{x_1,y}_t-Y^{x_2,y}_t)=A(Y^{x_1,y}_t-Y^{x_2,y}_t)dt+\left[g(x_1, Y^{x_1,y}_t)-g(x_2, Y^{x_2,y}_t)\right]dt.
%\end{align*}
%By Young's inequality, it is easy to see
%\begin{align*}
%&\frac{d}{dt}\|Y^{x_1,y}_t-Y^{x_2,y}_t\|^2\\
%=&2\|Y^{x_1,y}_t-Y^{x_2,y}_t\|^2_1+2\langle g(x_1, Y^{x_1,y}_t)-g(x_2, Y^{x_2,y}_t), Y^{x_1,y}_t-Y^{x_2,y}_t\rangle\\
%\leq&-2\lambda_1\|Y^{x_1,y}_t-Y^{x_2,y}_t\|^2\!\!+2L_g \|Y^{x_1,y}_t-Y^{x_2,y}_t\|^2\\
%&+C\|x_1-x_2\|\|Y^{x_1,y}_t-Y^{x_2,y}_t\|\\
%\leq&-(\lambda_1-L_g)\lambda_1\|Y^{x_1,y}_t-Y^{x_2,y}_t\|^2+C\|x_1-x_2\|^2.
%\end{align*}
%Then the compare theorem yields
%\begin{align*}
%\sup_{t\geq 0}\|Y^{x_1,y}_t-Y^{x_2,y}_t\|^2\leq \int^{\infty}_{0}e^{-(\lambda_1-L_g)s}ds\|x_1-x_2\|^2\leq (\lambda_1-L_g)^{-1}\|x_1-x_2\|^2.
%\end{align*}
%The proof is complete.
%\end{proof}

For any fixed $x\in H$ and $\mu\in \mathscr{P}_p$, we consider the following frozen equation:
\begin{eqnarray}\label{FZE}
\left\{ \begin{aligned}
&dY_{t}=\left[AY_{t}+G(x,\mu,Y_{t})\right]dt+d Z_{t},\\
&Y_{0}=y\in H.
\end{aligned} \right.
\end{eqnarray}
Since $G(x,\mu,\cdot)$ is Lipshcitz continuous,
it is easy to prove that equation \eref{FZE} has a unique mild solution denoted by $Y_{t}^{x,\mu,y}$, which is a time homogeneous Markovian process. By a straightforward computation, it is easy to prove
\begin{eqnarray}
\sup_{t\geq 0}\left(\mathbb{E}|Y_{t}^{x,\mu,y}|^{p}\right)^{1/p}\leq C\left[1+|x|+|y|+\left(\mu|\cdot|^{p}\right)^{1/p}\right].\label{EFY}
\end{eqnarray}

Let $P^{x,\mu}_t$ be the transition semigroup of $Y_{t}^{x,\mu,y}$, that is, for any bounded measurable function $\varphi$ on $H$ and $t \geq 0$,
\begin{align*}
P^{x,\mu}_t \varphi(y)= \mathbb{E} \varphi(Y_{t}^{x,\mu,y}), \quad y \in H.
\end{align*}
The asymptotic behavior of $P^{x,\mu}_t$ has been studied in many literatures, by a minor revision in \cite[Lemma 3.3]{BYY} , we have the
the following result :
\begin{proposition}\label{ergodicity}
For any $x\in H$ and $\mu\in\mathscr{P}_p$, $\{P^{x,\mu}_t\}_{t\geq 0}$ admits a unique invariant measure $\nu^{x,\mu}$. Moreover, for any $t>0$,
\begin{align*}
\left| \mathbb{E} F(x,\mu,Y_{t}^{x,\mu,y})-\int_{H}F(x,\mu,z)\nu^{x,\mu}(dz)\right|
\leq C\left[1+ |x| + \left(\mu(|\cdot|^p)\right)^{1/p} +|y |\right]e^{-(\lambda_1-L_G)t},
\end{align*}
where $C$ is a positive constant which is independent of $t$.
\end{proposition}

Now, we define $\bar{F}(x,\mu):=\int_{H}F(x,\mu,y)\nu^{x,\mu}(dy)$. Let $\bar{X}$ be the solution of the corresponding averaged equation:
\begin{equation}\left\{\begin{array}{l}
\displaystyle d\bar{X}_{t}=\left[A\bar{X}_{t}+\bar{F}(\bar{X}_{t},\mathscr{L}_{\bar{X}_t})\right]dt+d L_{t},\\
\bar{X}_{0}=\xi\in H. \end{array}\right. \label{1.3}
\end{equation}
The well-posedness of equation \eref{1.3} is the following:
\begin{theorem} \label{barX}
Equation \eref{1.3} exists a unique mild solution $\bar{X}_{t}$ which satisfies
\begin{align}
\bar X_t=e^{tA}\xi+\int^t_0e^{(t-s)A} \bar F(\bar X_s,  \mathscr{L}_{\bar X_s})ds+\int^t_0 e^{(t-s)A}dL_s.\label{3.22}
\end{align}
%Moreover, for any $x\in H$, $T>0$ and $1\leq p< \alpha$, there exists a constant $C_{p,T}>0$ such that
%\begin{align} \label{Control X}
%\mathbb{E}\left(\sup_{0\leq t\leq T}\|\bar X_{t} \|^p\right)+\EE\int^T_0 \frac{\|\bar X_{t} \|^2_1}{(\|\bar X_{t} \|^2+1)^{1-p/2}}ds
%\leq  C_{p,T}(1+ \|x\|^p).
%\end{align}
\end{theorem}
\begin{proof}
It is sufficient to check that the $ \bar F(\cdot,\cdot)$ is Lipschitz continuous, then the well-posedness can be easily obtained by following the procedures in Theorem \ref{main result 1}.

Note that it is easy to prove that  $x_1,x_2\in H$, $\mu_1,\mu_2\in \mathscr{P}_{p}$, we have
\begin{eqnarray}
\sup_{t\geq 0,y\in H}|Y^{x_1,\mu_1,y}_t-Y^{x_2,\mu_2,y}_t|\leq C|x_1-x_2|+C\mathbb{W}_p(\mu_1,\mu_2),\label{IncY}
\end{eqnarray}
Then by \eref{IncY} and Proposition \ref{ergodicity}, we get
\begin{align*}
|\bar{F}(x_1,\mu_1)-\bar{F}(x_2,\mu_2)|=&\left|\int_{H} F(x_1,\mu_1,z)\nu^{x_1,\mu_1}(dz)-\int_{H} F(x_2,\mu_2,z)\nu^{x_2,\mu_2}(dz)\right|\\
\leq&\left|\int_{H} F(x_1,\mu_1,z)\nu^{x_1,\mu_1}(dz)-\EE F(x_1, \mu_1,Y^{x_1,\mu_1,y}_t)\right|\\
&+\left|\EE F(x_1, \mu_1,Y^{x_1,\mu_1,y}_t)-\EE F(x_2, \mu_2,Y^{x_2,\mu_2,y}_t)\right|\\
&+\left|\EE F(x_2, \mu_2,Y^{x_2,\mu_2,y}_t)-\int_{H} F(x_2,\mu_2, z)\nu^{x_2,\mu_2}(dz)\right|\\
\leq&C\left[1+|x_1|+|x_2|+ \left(\mu_1(|\cdot|^p)\right)^{1/p}+ \left(\mu_2(|\cdot|^p)\right)^{1/p}+|y|\right]e^{-(\lambda_1-L_F) t}\\
&+C\left(|x_1-x_2|+\mathbb{W}_p(\mu_1,\mu_2)+\EE|Y^{x_1,\mu_1,y}_t-Y^{x_2,\mu_2,y}_t|\right)\\
\leq&C\left[1+|x_1|+|x_2|+ \left(\mu_1(|\cdot|^p)\right)^{1/p}+ \left(\mu_2(|\cdot|^p)\right)^{1/p}+|y|\right]e^{-(\lambda_1-L_F) t}\\
&+C|x_1-x_2|+C\mathbb{W}_p(\mu_1,\mu_2).
\end{align*}
Hence by letting $t\rightarrow \infty$, it follows
$$
|\bar{F}(x_1,\mu_1)-\bar{F}(x_2,\mu_2)|\leq C|x_1-x_2|+C\mathbb{W}_p(\mu_1,\mu_2).
$$
The proof is complete.
\end{proof}
\begin{remark}
Under the condition \eref{BDC}, it is easy to check that
\begin{eqnarray}
\sup_{x\in H}|\bar{F}(x,\mu)|\leq C\left[1+\left(\mu(|\cdot|^p)\right)^{1/p}\right].\label{barBDC}
\end{eqnarray}
\end{remark}
%By a standard argument (see e.g., \cite[Lemma 3.9]{SZ}), it is easy to see that the averaged coefficient $\bar{B}$ is Lipschitz, i.e.,
%$$
%\left|\bar{B}(x_1,\mu_1)-\bar{B}(x_2,\mu_2)\right|\leq C(|x_1-x_2|+\mathbb{W}_p(\mu_1,\mu_2)).
%$$
%Then the averaged equation \eref{AE} exists a unique mild solution $\bar{X}_{t}$.
%By the same argument in Lemmas \ref{PMY} and \ref{COX}, we have the following estimates:
%\begin{lemma} \label{barX}
%For any $x,y\in H$, $1\leq p<\alpha$ and $T>0$, there exist constants $C_{p,T}, C_{p,T,R}>0$ such that
%\begin{eqnarray} \label{Control barX}
%\sup_{t\in [0,T]}\mathbb{E}|\bar X_{t}|^p\leq  C_{p,T}(1+ |x| ^p),
%\end{eqnarray}
%\begin{eqnarray} \label{Control barX1}
%\mathbb{E}\left[\int^{T}_0|\bar X_t-\bar X_{t(\delta)}|dt\right]^{p}\leq C_{p,T}(|x|^p+1)\delta^{\frac{p\theta}{2}}.
%\end{eqnarray}
%\end{lemma}

\subsection{The proof of Theorem \ref{main result 2}} \label{sub 3.3}

In this subsection, we will give the detailed proof of Theorem \ref{main result 2}.

\begin{proof} It is easy to see that
\begin{eqnarray*}
X_{t}^{\vare}-\bar{X}_{t}=\!\!\!\!\!\!\!\!&&\int^t_0 e^{(t-s)A}\big[F(X^{\varepsilon}_s, \mathscr{L}_{X^{\vare}_s},Y^{\varepsilon}_s)-\bar F(\bar X_{s},\mathscr{L}_{\bar{X}_s})\big]ds\\
=\!\!\!\!\!\!\!\!&&\int^t_0 e^{(t-s)A}\big[F(X^{\varepsilon}_s, \mathscr{L}_{X^{\vare}_s},Y^{\varepsilon}_s)-F(X^{\varepsilon}_{s(\delta)}, \mathscr{L}_{X^{\vare}_{s(\delta)}},\hat Y^{\varepsilon}_s)\big]ds\\
&&+\int^t_0 e^{(t-s)A}\big[\bar F(X^{\vare}_{s},\mathscr{L}_{X^{\vare}_{s}})-\bar F(\bar X_{s},\mathscr{L}_{\bar{X}_{s}})\big]ds\\
&&+\int^t_0 e^{(t-s)A}\big[\bar F(X^{\vare}_{s(\delta)},\mathscr{L}_{X^{\vare}_{s(\delta)}})-\bar F(X^{\vare}_{s},\mathscr{L}_{X^{\vare}_{s}})\big]ds\\
&&+\int^t_0 e^{(t-s)A}\big[F(X^{\varepsilon}_{s(\delta)}, \mathscr{L}_{X^{\vare}_{s(\delta)}},\hat Y^{\varepsilon}_s)-\bar F(X^{\vare}_{s(\delta)},\mathscr{L}_{X^{\vare}_{s(\delta)}})\big]ds.
\end{eqnarray*}
Then by Lemmas \ref{L3.9} and \ref{DEY}, we have
\begin{eqnarray*}
&&\left[ \mathbb{E}\left(\sup_{t\in [0,T]}|X_{t}^{\vare}-\bar{X}_{t}|^{m}\right)\right]^{1/m}\\
\leq\!\!\!\!\!\!\!\!&&\left[\EE\left|\int^T_0\left|F(X^{\varepsilon}_s, \mathscr{L}_{X^{\vare}_s},Y^{\varepsilon}_s)-F(X^{\varepsilon}_{s(\delta)}, \mathscr{L}_{X^{\vare}_{s(\delta)}},\hat Y^{\varepsilon}_s)\right|ds\right|^m\right]^{1/m}\\
&&+\left[\EE\left|\int^T_0\left|\bar F(X^{\vare}_{s},\mathscr{L}_{X^{\vare}_{s}})-\bar F(\bar X_{s},\mathscr{L}_{\bar{X}_{s}})\right|ds\right|^m\right]^{1/m}\\
&&+\left[\EE\left|\int^T_0\left|\bar F(X^{\vare}_{s(\delta)},\mathscr{L}_{X^{\vare}_{s(\delta)}})-\bar F(X^{\vare}_{s},\mathscr{L}_{X^{\vare}_{s}})\right|ds\right|^m\right]^{1/m}\\
&&+\left[\EE\left|\sup_{t\in [0,T]}\int^t_0 e^{(t-s)A}\big[F(X^{\varepsilon}_{s(\delta)}, \mathscr{L}_{X^{\vare}_{s(\delta)}},\hat Y^{\varepsilon}_s)-\bar F(X^{\vare}_{s(\delta)},\mathscr{L}_{X^{\vare}_{s(\delta)}})\big]ds\right|^m\right]^{1/m}\\
\leq\!\!\!\!\!\!\!\!&&C\int^T_0\EE\left[\left|X^{\varepsilon}_s-X^{\varepsilon}_{s(\delta)}\right|^m\right]^{1/m}+\left[\EE\left|Y^{\varepsilon}_s-\hat Y^{\varepsilon}_s\right|^m\right]^{1/m}+\left[\EE|X^{\varepsilon}_s-X^{\varepsilon}_{s(\delta)}|^p\right]^{1/p}ds\\
&&+C\int^T_0\left(\EE|X_{s}^{\vare}-\bar{X}_{s}|^{m}\right)^{1/m}+\left(\EE|X_{s}^{\vare}-\bar{X}_{s}|^{p}\right)^{1/p}ds\\
&&+\left[\EE\left|\sup_{t\in [0,T]}\int^t_0 e^{(t-s)A}\big[F(X^{\varepsilon}_{s(\delta)}, \mathscr{L}_{X^{\vare}_{s(\delta)}},\hat Y^{\varepsilon}_s)-\bar F(X^{\vare}_{s(\delta)},\mathscr{L}_{X^{\vare}_{s(\delta)}})\big]ds\right|^m\right]^{1/m}\\
\leq\!\!\!\!\!\!\!\!&&C\int^T_0\left(\EE|X_{t}^{\vare}-\bar{X}_{t}|^{m}\right)^{1/m} dt+C_{T}\left[\left(\EE|\xi|^m\right)^{1/m}+1\right]\delta^{\frac{\theta}{2}}\\
&&+\left[\EE\left|\sup_{t\in [0,T]}\int^t_0 e^{(t-s)A}\big[F(X^{\varepsilon}_{s(\delta)}, \mathscr{L}_{X^{\vare}_{s(\delta)}},\hat Y^{\varepsilon}_s)-\bar F(X^{\vare}_{s(\delta)},\mathscr{L}_{X^{\vare}_{s(\delta)}})\big]ds\right|^m\right]^{1/m}.
\end{eqnarray*}
Then by Gronwall's inequality we obtain that
\begin{eqnarray}
\left[ \mathbb{E}\left(\sup_{t\in [0,T]}|X_{t}^{\vare}-\bar{X}_{t}|^{m}\right)\right]^{1/m}
\leq\!\!\!\!\!\!\!\!&&C_{T}\left[\left(\EE|\xi|^m\right)^{1/m}+1\right]\delta^{\frac{\theta}{2}}+C_{T}J(T,m,\vare,\delta),\label{F3.17}
\end{eqnarray}
where
$$J(T,m,\vare,\delta):=\left\{\EE\left|\sup_{t\in [0,T]}\int^t_0 e^{(t-s)A}\big[F(X^{\varepsilon}_{s(\delta)}, \mathscr{L}_{X^{\vare}_{s(\delta)}},\hat Y^{\varepsilon}_s)-\bar F(X^{\vare}_{s(\delta)},\mathscr{L}_{X^{\vare}_{s(\delta)}})\big]ds\right|^m\right\}^{1/m}.$$
Then by \eref{BDC} and \eref{barBDC}, we have
\begin{eqnarray*}
&&J(T,m,\vare,\delta)\nonumber\\
\leq\!\!\!\!\!\!\!\!&&\left\{\EE\left|\sup_{t\in [0,T]}\int^t_0 e^{(t-s)A}\big[F(X^{\varepsilon}_{s(\delta)}, \mathscr{L}_{X^{\vare}_{s(\delta)}},\hat Y^{\varepsilon}_s)-\bar F(X^{\vare}_{s(\delta)},\mathscr{L}_{X^{\vare}_{s(\delta)}})\big]ds\right|^2\right\}^{1/2}\\
\leq\!\!\!\!\!\!\!\!&&C\left\{\EE\sup_{t\in [0,T]}\left|
\sum^{t(\delta)-1}_{l=0}e^{(t-(l+1)\delta)A}\int_{l\delta}^{(l+1)\delta}\!\!\!\!e^{((l+1)\delta-s)A}\big[F(X^{\varepsilon}_{l\delta}, \mathscr{L}_{X^{\vare}_{l\delta}},\hat Y^{\varepsilon}_s)-\bar F(X^{\vare}_{l\delta},\mathscr{L}_{X^{\vare}_{l\delta}})\big]ds\right|^2\right\}^{1/2}\nonumber\\
&&+C\left\{\EE\sup_{t\in [0,T]}\left|\int_{t(\delta)}^{t}e^{(t-s)A}\big[F(X^{\varepsilon}_{t(\delta)}, \mathscr{L}_{X^{\vare}_{t(\delta)}},\hat Y^{\varepsilon}_s)-\bar F(X^{\vare}_{t(\delta)},\mathscr{L}_{X^{\vare}_{t(\delta)}})\big] ds\right|^{2}\right\}^{1/2}\nonumber\\
\leq\!\!\!\!\!\!\!\!&&C\sum_{l=0}^{[T/\delta]-1}
\left\{\EE\left|\int_{l\delta}^{(l+1)\delta}e^{((l+1)\delta-s)A}\big[F(X^{\varepsilon}_{l\delta}, \mathscr{L}_{X^{\vare}_{l\delta}},\hat Y^{\varepsilon}_s)-\bar F(X^{\vare}_{l\delta},\mathscr{L}_{X^{\vare}_{l\delta}})\big]ds\right|^{2}\right\}^{1/2}+C_T\delta\nonumber\\
\leq\!\!\!\!\!\!\!\!&&\frac{C_T\vare}{\delta}\max_{0\leq l\leq[T/\delta]-1}\left[\EE\left|\int_{0}^{\frac{\delta}{\vare}}e^{(\delta-s\vare)A}
F(X_{l\delta}^{\vare},\mathscr{L}_{X^{\vare}_{l\delta}},\hat{Y}_{s\vare+l\delta}^{\vare})-\bar{F}(X_{l\delta}^{\vare},\mathscr{L}_{X^{\vare}_{l\delta}})ds\right|^2\right]^{1/2}+C_T\delta\nonumber\\
\leq\!\!\!\!\!\!\!\!&&\frac{C_{T}\vare}{\delta}\max_{0\leq l\leq[T/\delta]-1}\left[\int_{0}^{\frac{\delta}{\vare}}
\int_{r}^{\frac{\delta}{\vare}}\Psi_{l}(s,r)dsdr\right]^{1/2}+C_T\delta,
\end{eqnarray*}
where for any $0\leq r\leq s\leq \frac{\delta}{\vare}$ and $l=0,1,\ldots, [T/\delta]-1$,
\begin{eqnarray*}
\Psi_{l}(s,r):=\!\!\!\!\!\!\!\!&&\mathbb{E}
\big\langle e^{(\delta-s\vare)A}\left[F(X_{l\delta}^{\vare},\mathscr{L}_{X^{\vare}_{l\delta}},\hat{Y}_{s\vare+l\delta}^{\vare})-\bar{F}(X_{l\delta}^{\vare},\mathscr{L}_{X^{\vare}_{l\delta}})\right],\nonumber\\
&&\quad\quad\quad e^{(\delta-r\vare)A}\left[F(X_{l\delta}^{\vare},\mathscr{L}_{X^{\vare}_{l\delta}},\hat{Y}_{r\vare+l\delta}^{\vare})-\bar{F}(X_{l\delta}^{\vare},\mathscr{L}_{X^{\vare}_{l\delta}})\right]\big\rangle.
\end{eqnarray*}
Refer to  Appendix, the following estimation holds:
\begin{eqnarray}
\Psi_{l}(s,r)\leq C_T\left[1+\left(\EE|\xi|^p\right)^{2/p}+\left(\EE|\eta|^p\right)^{2/p}\right]e^{-(\lambda_1-L_G)(s-r)},\quad \forall l=0,1,\ldots, [T/\delta]-1,\label{F3.27}
\end{eqnarray}
As a result, we obtain
\begin{eqnarray}
J(T,m,\vare,\delta)\leq C_{T}\left[1+\left(\EE|\xi|^p\right)^{1/p}+\left(\EE|\eta|^p\right)^{1/p}\right]\left(\frac{\vare^{1/2}}{\delta^{1/2}}+\delta\right).\label{J4}
\end{eqnarray}
By \eref{F3.17} and \eref{J4}, it is easy to see
\begin{eqnarray*}
\left[\EE\left(\sup_{t\in [0,T]}|X_{t}^{\vare}-\bar{X}_{t}|^m\right)\right]^{1/m}
\leq\!\!\!\!\!\!\!\!&&C_{T}\left[1+\left(\EE|\xi|^m\right)^{1/m}+\left(\EE|\eta|^{p}\right)^{1/p}\right]\left(\delta^{\frac{\theta}{2}}+\frac{\vare^{1/2}}{\delta^{1/2}}+\delta\right).
\end{eqnarray*}
Finally, taking $\delta=\vare^{\frac{1}{1+\theta}}$, then we obtain that
$$
\left[\EE\left(\sup_{t\in [0,T]}|X_{t}^{\vare}-\bar{X}_{t}|^m\right)\right]^{1/m}
\leq C_{T}\left[1+\left(\EE|\xi|^m\right)^{1/m}+\left(\EE|\eta|^{p}\right)^{1/p}\right]\vare^{\frac{\theta}{2(1+\theta)}}.
$$
The proof is complete.
\end{proof}
\begin{remark}
Note that under the condition \eref{BDC}, it follows
$$\sup_{t\in [0,T]}|X_{t}^{\vare}-\bar{X}_{t}|\leq \int^T_0|F(X^{\varepsilon}_t, \mathscr{L}_{X^{\vare}_t},Y^{\varepsilon}_t)-\bar F(\bar X_{t},\mathscr{L}_{\bar{X}_t})|dt\leq C_T.$$
Thus for any $T>0$ and $k\geq 1$, we have
$$
\lim_{\vare\rightarrow 0} \mathbb{E}\left[\sup_{t\in [0,T]}|X_{t}^{\vare}-\bar{X}_{t}|^{k}\right]=0.
$$
\end{remark}

\section{Appendix}

In this section, we give the detailed proof of \eref{F3.27}.\\
\textbf{The proof of \eref{F3.27}: } For any $\mu\in\mathscr{P}_p$ and random variables $\tilde{\xi},\tilde{\eta}\in\mathcal{F}_s$, we consider the following equation
\begin{eqnarray*}
d\tilde{Y}^{\vare,s,\tilde{\xi},\mu,\tilde{\eta}}_t=\frac{1}{\vare}A\tilde{Y}^{\vare,s,x,\mu,y}_t dt+\frac{1}{\vare}G(\tilde{\xi},\mu,\tilde{Y}^{\vare,s,\tilde{\xi},\mu,\tilde{\eta}}_t)dt+\frac{1}{\vare^{1/\alpha}}dZ_t,\quad t\geq s.
\end{eqnarray*}
with $\tilde{Y}^{\vare,s,\tilde{\xi},\mu,\tilde{\eta}}_s=\tilde{\eta}$. Then by the construction of $\hat{Y}_{t}^{\vare}$, for any $l\in \mathbb{N}_{+}$, we have
$$
\hat{Y}_{t}^{\vare}=\tilde Y^{\vare,l\delta,X_{k\delta}^{\vare},\mathscr{L}_{X^{\vare}_{k\delta}},\hat{Y}_{k\delta}^{\vare}}_t,\quad t\in[l\delta,(l+1)\delta],
$$
which implies
\begin{eqnarray*}
\Psi_{l}(s,r)=\!\!\!\!\!\!\!\!&&\mathbb{E}
\left\langle e^{(\delta-s\vare)A}\left[F(X_{l\delta}^{\vare},\mathscr{L}_{X^{\vare}_{l\delta}},\tilde Y^{\vare,l\delta,X_{k\delta}^{\vare},\mathscr{L}_{X^{\vare}_{k\delta}},\hat{Y}_{k\delta}^{\vare}}_{s\vare+l\delta})-\bar{F}(X_{l\delta}^{\vare},\mathscr{L}_{X^{\vare}_{l\delta}})\right]\right.,\nonumber\\
&&\quad\quad\quad \left.e^{(\delta-r\vare)A}\left[F(X_{l\delta}^{\vare},\mathscr{L}_{X^{\vare}_{l\delta}},\tilde Y^{\vare,l\delta,X_{k\delta}^{\vare},\mathscr{L}_{X^{\vare}_{k\delta}},\hat{Y}_{k\delta}^{\vare}}_{r\vare+l\delta})-\bar{F}(X_{l\delta}^{\vare},\mathscr{L}_{X^{\vare}_{l\delta}})\right]\right\rangle.\nonumber\\
\end{eqnarray*}

Note that for any fixed $x,y\in H$ and $\mu\in \mathscr{P}_p$, $\tilde Y^{\vare, l\delta, x, \mu, y}_{s\vare+l\delta}$ is independent of $\mathcal{F}_{l\delta}$. $X_{l\delta}^{\vare}$ and $\hat Y_{l\delta}^{\vare}$ are $\mathcal{F}_{l\delta}$-measurable, we have
\begin{eqnarray*}
\Psi_{l}(s,r)=\!\!\!\!\!\!\!\!&&\EE\Bigg\{\mathbb{E}\left[
\left \langle e^{(\delta-s\vare)A}\left(F\left ( X_{l\delta}^{\vare},\mathscr{L}_{X^{\vare}_{l\delta}},\tilde {Y}^{\vare, l\delta, X_{l\delta}^{\vare},\mathscr{L}_{X^{\vare}_{l\delta}}, \hat Y_{l\delta}^{\vare}}_{s\vare+l\delta}\right )-\bar{F}\left (X_{l\delta}^{\vare},\mathscr{L}_{X^{\vare}_{l\delta}}\right )\right),\right.\right.\\
&&\quad\quad\quad \left.\left.e^{(\delta-r\vare)A}\left(F\left (X_{l\delta}^{\vare},\mathscr{L}_{X^{\vare}_{l\delta}},\tilde{Y}^{\vare, l\delta, X_{l\delta}^{\vare}, \mathscr{L}_{X^{\vare}_{l\delta}},\hat Y_{k\delta}^{\vare}}_{r\vare+l\delta}\right )-\bar{F}\left (X_{l\delta}^{\vare},\mathscr{L}_{X^{\vare}_{l\delta}}\right )\right)\right \rangle \Big|\mathcal{F}_{l\delta}\right](\omega)\Bigg\}\\
=\!\!\!\!\!\!\!\!&&\mathbb{E}\Big\{\mathbb{E}
\big\langle e^{(\delta-s\vare)A}\left[F(x,\mu,\tilde{Y}^{\vare,l\delta,x,\mu,y}_{s\vare+l\delta})-\bar{F}(x,\mu)\right],\nonumber\\
&&\quad\quad\quad e^{(\delta-r\vare)A}\left[F(x,\mu,\tilde{Y}^{\vare,l\delta,x,\mu,y}_{r\vare+l\delta})-\bar{F}(x,\mu)\right]\big\rangle\mid_{(x,\mu,y)=(X_{l\delta}^{\vare},\mathscr{L}_{X^{\vare}_{l\delta}},\hat{Y}_{l\delta}^{\vare})}\Big\}.
\end{eqnarray*}
By the construction of the process $\{\tilde{Y}^{\vare,s,x,\mu,y}_t\}_{t\geq s}$, which has the integrated form:
\begin{eqnarray}
\tilde{Y}^{\vare,l\delta,x,\mu,y}_{s\vare+l\delta}\!=\!\!\!\!\!\!\!\!&&y\!+\!\frac{1}{\vare}\int^{s\vare+l\delta}_{l\delta}\!\! A\tilde{Y}^{\vare,l\delta,x,\mu,y}_rdr+\!\frac{1}{\vare}\int^{s\vare+l\delta}_{l\delta}\!\! G\left (x,\mu,\tilde{Y}^{\vare,l\delta,x,\mu,y}_r\right )dr\!+\!\frac{1}{\vare^{1/\alpha}}\int^{s\vare+l\delta}_{l\delta}dZ_r\nonumber\\
=\!\!\!\!\!\!\!\!&&y\!+\int^{s}_{0}\!\! A\tilde{Y}^{\vare,l\delta,x,\mu,y}_{r\vare+l\delta}dr+\int^{s}_{0}\!\! G\left (x,\mu,\tilde{Y}^{\vare,l\delta,x,\mu,y}_{r\vare+l\delta}\right )dr\!+\int^{s}_{0}d\hat{Z}_r,\label{E3.15}
\end{eqnarray}
where $\{\hat Z_t:=\frac{1}{\vare^{1/\alpha}}(Z_{t\vare+l\delta}-Z_{l\delta})\}_{t\geq 0}$, which is also a  cylindrical $\alpha$-stable process. Recall the solution of the frozen equation satisfies
\begin{eqnarray}
Y_{s}^{x, \mu,y}= y\!+\int^{s}_{0}\!\! AY_{r}^{x, \mu,y}dr+\int^{s}_{0}\!\!G\left (x,\mu,Y_{r}^{x, \mu,y}\right )dr\!+\int^{s}_{0}dZ_r.  \label{E3.16}
\end{eqnarray}
The uniqueness of the solutions of equation (\ref{E3.15}) and equation (\ref{E3.16}) implies
that the distribution of $\{\tilde Y^{\vare, l\delta, x,\mu,y}_{s\vare+l\delta}\}_{0\leq s\leq \delta/\vare}$
coincides with the distribution of $\{Y_{s}^{x, \mu,y}\}_{0\leq s\leq \delta/\vare}$.

%For fixed $x,y\in H$ and $\mu\in \mathscr{P}_p$, define
%$$
%\tilde{\mathcal{F}}_s:=\sigma\{ Y_{u}^{x,\mu,y},u\leq s\}.
%$$
Hence by Markov property, Proposition \ref{ergodicity}, \eref{BDC}, \eref{EAY} and \eref{EFY}, we have
\begin{eqnarray*}
\Psi_{l}(s,r)
=\!\!\!\!\!\!\!\!&&\mathbb{E}\Big\{\mathbb{E}
\big\langle e^{(\delta-s\vare)A}\left[F(x,\mu,Y^{x,\mu,y}_s)-\bar{F}(x,\mu)\right],\nonumber\\
&&\quad\quad\quad e^{(\delta-r\vare)A}\left[F(x,\mu,Y^{x,\mu,y}_{r})-\bar{F}(x,\mu)\right]\big\rangle\mid_{(x,\mu,y)=(X_{l\delta}^{\vare},\mathscr{L}_{X^{\vare}_{l\delta}},\hat{Y}_{l\delta}^{\vare})}\Big\}\nonumber\\
=\!\!\!\!\!\!\!\!&&\mathbb{E}\left\{\EE\left[\left\langle e^{(\delta-s\varepsilon)A}
\EE \big[F(x,\mu,Y^{x,\mu,y}_{s})-\bar{F}(x,\mu)\mid \mathcal{F}_{r}\big],\right.\right.\right.\\
&&\quad\quad\quad \left.\left.\left.e^{(\delta-r\varepsilon)A}\big(F(x,\mu,Y^{x,\mu,y}_{r})-\bar{F}(x,\mu)\big)\right\rangle\right]\mid_{(x,\mu,y)=(X_{l\delta}^{\varepsilon},\mathscr{L}_{X^{\vare}_{l\delta}},\hat{Y}^{\vare}_{l\delta})}\right\}\nonumber\\
\leq\!\!\!\!\!\!\!\!&&C\mathbb{E}\left\{\EE\left[|\EE F(x,\mu,Y^{x,\mu,z}_{s-r})-\bar{F}(x,\mu)|\mid_{z=Y_{r}^{x,\mu,y}}\right]\mid_{(x,\mu,y)=(X_{l\delta}^{\varepsilon},\mathscr{L}_{X^{\vare}_{l\delta}},\hat Y_{l\delta}^{\varepsilon})}\right\}\nonumber\\
\leq\!\!\!\!\!\!\!\!&&C\mathbb{E}\left\{\EE\left[1+|x|+\left(\mu(|\cdot|^p)\right)^{1/p}+|Y^{x,\mu,y}_{r}|\right]\left[1+\left(\mu(|\cdot|^p)\right)^{1/p}\right]\right.\\
&&\quad\quad\quad\quad \left.\cdot e^{-(\lambda_1-L_F)(s-r)}\mid_{(x,\mu,y)=(X_{l\delta}^{\varepsilon},\mathscr{L}_{X^{\vare}_{l\delta}},\hat Y_{l\delta}^{\varepsilon})}\right\}\nonumber\\
\leq\!\!\!\!\!\!\!\!&&C\mathbb{E}\left(1+|X_{l\delta}^{\varepsilon}|+\left(\EE|X^{\vare}_{l\delta}|^p\right)^{1/p}+|\hat Y_{l\delta}^{\varepsilon}|\right)\left[1+\left(\EE|X^{\vare}_{l\delta}|^p\right)^{1/p}\right]e^{-(\lambda_1-L_F)(s-r)}\nonumber\\
\leq\!\!\!\!\!\!\!\!&&C_T\left[1+\left(\EE|\xi|^p\right)^{2/p}+\left(\EE|\eta|^p\right)^{2/p}\right]e^{-(\lambda_1-L_G)(s-r)},
\end{eqnarray*}
which completes the proof.

\vspace{0.3cm}
\textbf{Acknowledgment}. This work is supported by the National Natural Science Foundation of China (11801233, 11771187, 11931004, 12090011), the QingLan Project and the Priority Academic Program Development of Jiangsu Higher Education Institutions.

\end{document}